\documentclass{article}

\usepackage{amsfonts}
\usepackage{cite}
\usepackage{setspace}
\usepackage{epsfig}

\usepackage{graphicx} 
\usepackage{booktabs} 
\usepackage{amsmath}
\usepackage{color}
\usepackage{float}
\usepackage[font=small,labelfont=bf]{caption}
\usepackage{amssymb}
\usepackage{amsthm}

\newtheorem{theorem}{Theorem}
\newtheorem{lemma}{Lemma}
\newtheorem{conjecture}{Conjecture}
\newtheorem{corollary}{Corollary}

\begin{document}

\title{Global Cycle Properties in Graphs with Large Minimum Clustering Coefficient}
\author{Adam Borchert, Skylar Nicol, Ortrud R. Oellermann\footnote{Supported by an NSERC grant CANADA}\\
 \small{Department of Mathematics and Statistics}\\
 \small{University of Winnipeg, Winnipeg MB, CANADA}\\
 \small{adamdborchert@gmail.com; skylarnicol93@gmail.com; o.oellermann@uwinnipeg.ca}}
\date{}

\maketitle

\begin{abstract}
Let $\cal P$ be a graph property. A graph $G$ is said to be {\em locally} $\cal P$ ({\em closed locally} $\cal{P}$) if the subgraph induced by the open neighbourhood (closed neighbourhood, respectively) of every vertex in $G$ has property $\cal P$.  The {\em clustering coefficient} of a vertex is the proportion of pairs of its neighbours that are themselves neighbours. The {\em minimum clustering coefficient} of $G$ is the smallest clustering coefficient among all vertices of $G$. Let $H$ be a subgraph of a graph $G$ and let $S\subseteq V(H)$. We say that $H$ is a \textit{strongly induced subgraph} of $G$ with {\em attachment set} $S$, if $H$ is an induced subgraph of $G$ and the vertices of $V(H)-S$ are not incident with edges that are not in $H$. A graph $G$ is {\em fully cycle extendable}  if every vertex of $G$ lies in a triangle and for every nonhamiltonian cycle $C$ of $G$, there is a cycle of length $|V(C)|+1$ that contains the vertices of $C$. A complete characterization, of those locally connected graphs with minimum clustering coefficient 1/2 and maximum degree at most $6$ that are fully cycle extendable, is given in terms of forbidden strongly induced subgraphs (with specified attachment sets).  Moreover, it is shown that all locally connected graphs with $\Delta \le 6$ and sufficiently large minimum clustering coefficient are weakly pancylic, thereby proving Ryj\'{a}\v{c}ek's conjecture for this class of graphs. \\

\medskip

\end{abstract}

\section{Introduction}

The evolution of the internet and the resulting large communication, information and social networks has significantly impacted recent advances in graph theory. The {\em clustering coefficient} of a user in a social network is the proportion of pairs of its friends that are themselves friends. The {\em mean clustering coefficient} for the network is the average clustering coefficient taken over all users. This concept was introduced by Watts and Strogatz \cite{WS98} to determine whether a given network is a `small world' network. It was reported in Ugander et. al \cite{UKBM11} that the Facebook graph is locally dense but globally sparse. Indeed the clustering coefficient of vertices of small degree is close to $1/2$ and steadily decreases with increasing degree. In this paper we focus on the global cycle structure of graphs that are locally sufficiently dense. Unless specified otherwise, all graphs under consideration are assumed to be connected and of order at least $3$.

For graph theory terminology not introduced here we follow \cite{bm}. Let $G$ be a graph and $v$ a vertex of $G$. Then the {\em neighbourhood} of $v$, denoted by $N(v)$, is the set of all vertices adjacent with $v$ and the closed neighbourhood of $v$ is the set $N[v] = N(v) \cup\{v\}$. If $S$ is a set of vertices in a graph, then the {\em subgraph induced} by $S$, denoted by $\langle S \rangle$, is the subgraph of $G$ with vertex set $S$ such that two vertices of $S$ are adjacent in $\langle S \rangle$ if and only if they are adjacent in $G$.  Let $\cal P$ be a graph property. A graph $G$ is said to be {\em locally}  $\cal P$  ({\em closed locally} $\mathcal{P}$), if the subgraph induced by the open neighbourhood (closed neighbourhood, respectively) of every vertex in $G$ has property $\cal P$.

The {\em clustering coefficient} of a vertex $v$ of degree $k \ge 2$, denoted by $\xi(v)$, is the ratio of the number of edges in the subgraph induced by $N(v)$ to the maximum number of edges in a $k$-vertex graph, i.e, $\xi(v) =|E(\langle N(v) \rangle)|/{k \choose 2}$. The {\em minimum clustering coefficient} of a graph $G$ with minimum degree $\delta(G) \ge 2$, denoted by $\xi(G)$, is the smallest clustering coefficient among all vertices of $G$. In the sequel we assume that all graphs under consideration have minimum degree at least $2$. If the minimum clustering coefficient of a class of graphs, is bounded by a constant $c$, $0 < c \le 1$, then these graphs are said to be {\em locally dense}. We observe that there are classes of graphs, although being locally dense, may be globally sparse with large diameter. For example, the class of graphs that are strong products (see \cite{IK}) of a path $P_n$ of order $n \ge 2$ and the complete graph $K_2$ of order $2$, denoted by $P_n \boxtimes K_2$, has minimum clustering coefficient at least $1/2$, is sparse since $|E(P_n \boxtimes K_2)|/{{2n} \choose 2} \rightarrow 0$  as $n \rightarrow \infty$ and has diameter $n-1$. In this paper we show that connected locally connected graphs with bounded maximum degree, for which the minimum clustering coefficient is at least 1/2, have a rich cycle structure.  In order to make these notions more precise we begin with pertinent definitions and relevant background on cycle properties in graphs that possess certain local properties.

Let $G$ be a graph of order $n$. Then $G$ is \emph{hamiltonian} if $G$ has a cycle of length $n$.  If, in addition, $G$ has a cycle of every length from 3 up to $n$, then $G$ is \emph{pancyclic}, see \cite{B}. An even stronger notion than pancyclicity is that of `full cycle extendability', introduced by Hendry \cite{H1}. A cycle $C$ in a graph $G$ is {\em extendable} if there exists a cycle $C'$ in $G$ that contains all the vertices of $C$ and one additional  vertex. The graph $G$ is {\em cycle extendable} if every nonhamiltonian cycle of $G$ is extendable. If, in addition, every vertex of $G$ lies on a 3-cycle, then $G$ is \emph{fully cycle extendable}. A condition weaker than pancyclicity has been the focus of a number of research articles.
 Let $g(G)$ and $c(G)$ denote, respectively, the length of a shortest cycle, called the {\em girth} of $G$, and a longest cycle, called the {\em circumference} of $G$. Then $G$ is called \emph{weakly pancyclic} if $G$ has a cycle of every length between $g(G)$ and $c(G)$. The problem of determining whether or not a graph has a hamiltonian cycle is called the {\em Hamilton Cycle Problem}.

In 1974, Chartrand and Pippert \cite{CP} initiated the study of locally connected graphs. Cycle properties of these graphs with bounded maximum degree have since been studied extensively - see for example \cite{AFOW13,ad,CGP,C,GOPS,K,H1,H2,OS,PS}. Another class of graphs that has been widely studied in relation to the Hamilton Cycle Problem, is the class of claw-free, i.e. $K_{1,3}$-free, graphs. These are precisely the graphs $G$ for which $\alpha(\langle N(v) \rangle ) \le 2$ for all $v \in V(G)$. Thus claw-free graphs can also be described in terms of a local property. The Hamilton Cycle Problem is NP-complete for both locally connected graphs, see \cite{GOPS}, and for claw-free graphs, see \cite{LCM}. The next result of Oberly and Sumner demonstrates the strength of combining these two local properties.

\begin{theorem} \label{oberly and sumner} \emph{\cite{OS}}
If $G$ is a connected, locally connected claw-free graph, then $G$ is hamiltonian.
\end{theorem}

Clark \cite{C} showed that connected, locally connected claw-free graph are in fact pancyclic and Hendry \cite{H1} observed that Clark had actually shown that these graphs are fully cycle extendable.
These results support Bondy's well-known `meta-conjecture' that almost any condition that guarantees that a graph has a Hamilton cycle actually guarantees much more about the cycle structure of the graph.

If, in Theorem \ref{oberly and sumner}, the claw-free condition is dropped, hamiltonicity is no longer guaranteed. In fact, Pareek and Skupi{\'e}n \cite{PS} observed that there exist infinitely many connected, locally hamiltonian graphs that are nonhamiltonian. However, Clark's result led Ryj\'{a}\v{c}ek to suspect that every locally connected graph has a rich cycle structure, even if it is not hamiltonian. He proposed the following conjecture (see \cite{WR}.)

 \begin{conjecture} \label{ryjacek}
 (Ryj\'{a}\v{c}ek) Every locally connected graph is weakly pancyclic.
 \end{conjecture}

 Ryj\'{a}\v{c}ek's conjecture seems to be very difficult to settle, so it is natural to consider weaker conjectures. This conjecture has been studied, for example, for locally traceable and locally hamiltonian graphs with maximum degree at most 5 and 6, respectively, see \cite{AFOW13}, neither of which need to be hamiltonian, see \cite{ad}. One may well ask whether there are local connectedness conditions that guarantee (global) hamiltonicity. One such result was obtained by Hasratian and Kachatrian \cite{HK}. Recall that a graph $G$ of order $n \ge 3$ has the Ore property if $\deg u + \deg v \ge n$ for all pairs $u,v$ of non-adjacent vertices $u$ and $v$ of $G$ and it has the Dirac property if $\deg v \ge n/2$ for all $v \in V(G)$. It was shown in \cite{HK} that closed locally Ore graphs are hamiltonian and this result was subsequently strengthened in \cite{A} to full cycle extendability for these graphs. Thus, in particular, all closed locally Dirac graphs are fully cycle extendable. If a graph is closed locally Dirac, then the clustering coefficient of each vertex is at least $1/2$. This prompts the question: can the  closed locally Dirac condition can be weakened to requiring the graph to have minimum local clustering coefficient at least 1/2 while still guaranteeing hamiltonicity of the graph? We show in this paper that locally connected graphs with minimum clustering coefficient at least $1/2$ need not be hamiltonian. Nevertheless our results demonstrate that these graphs have a rich cycle structure and they lend support to Ryj\'{a}\v{c}ek's conjecture.

 The Hamilton Cycle Problem for graphs with small maximum degree remains difficult, even when additional structural properties are imposed on the graph. For example, the Hamilton Cycle Problem is NP-complete for bipartite planar graphs with $\Delta \leq 3$ (see \cite{ANS}), for $r$-regular graphs for any fixed $r$ (see \cite{P}) and even for planar cubic 3-connected claw-free graphs (see \cite{LCM}). However, some progress has been made for locally connected graphs with small maximum degree. The first result in this connection was obtained in \cite{CP} where it was shown that every connected, locally connected graph with maximum degree $\Delta (G) \le 4$  is either hamiltonian or isomorphic to the complete $3$-partite graph $K_{1,1,3}$. All connected, locally connected graphs with maximum degree at most $4$ are described in \cite{GOPS}. Apart from $K_{1,1,3}$, all of these graphs are actually fully cycle extendable. Since $K_{1,1,3}$ is weakly pancyclic,  Ryj\'{a}\v{c}ek's conjecture holds for locally connected graphs with maximum degree at most 4.

Global cycle properties of connected, locally connected graphs with maximum degree 5 were investigated in \cite{GOPS,H2,K}. Collectively  these results imply that every connected, locally connected graph with $\Delta =5$ and $\delta \geq 3$ is fully cycle extendable.

A graph $G$ is {\em locally isometric} if $\langle N(v) \rangle$ is an isometric subgraph (distance preserving subgraph) of $G$ for all $v \in V(G)$. It was shown in \cite{BNO14} that the Hamilton Cycle Problem is NP-complete even for locally isometric graphs with maximum degree $8$.  Nevertheless, it was shown in \cite{BNO14} that Ryj\'{a}\v{c}ek's conjecture holds for all locally isometric graphs with maximum degree 6 and without true twins (i.e., pairs of vertices having the same closed neighbourhood);  these graphs are in fact fully cycle extendable.

Let $x,y$ be vertices of a graph $G$ and $S$ a set of vertices of $G$. If $x$ is adjacent with $y$ (or every vertex of $S$) we write $x \sim y$ (or $x \sim S$, respectively). We use $x \nsim S$ to indicate that $x$ is not adjacent with any vertex of $S$.

\section{Cycle Structure in Graphs with Minimum Clustering Coefficient at least $\frac{1}{2}$}

In this section we obtain a structural characterization of those connected, locally connected graphs with minimum clustering coefficient $1/2$ and maximum degree at most 6 that are fully cycle extendable. To state this characterization we introduce some useful definitions. Let $H$ be a subgraph of a graph $G$ and let $S\subseteq V(H)$. We say that $H$ is a \textit{strongly induced subgraph} of $G$ with {\em attachment set} $S$ if $H$ is an induced subgraph of $G$ and the vertices of $V(H)-S$ are not incident with edges that do not belong to $H$, i.e., only the vertices of $S$ may be incident with edges of $G$ that are not in $H$.

Let $G$ be a connected, locally connected graph with minimum clustering coefficient $\xi(G) \ge 1/2$.  If $G$ has maximum degree $\Delta=2$ or $3$, then $G$ is isomorphic to $K_3$ or either $K_4$ or $K_4-e$, respectively (where $e$ is an edge of $K_4$). So $G$ is fully cycle extendable. If $\Delta = 4$, then it is readily seen that $G$ is fully cycle-extendable unless $G \cong K_2 + \overline{K}_3$.

For the remainder of this section we focus on graphs with $\Delta \ge 5$.  Let $C=v_0v_1v_2\ldots v_{t-1} v_0$ be a $t$-cycle in a graph $G$. If $i \ne j$ and $\{i,j\}\subseteq \{0,1,\ldots, t-1\}$, then $v_i\overrightarrow{C}v_j$ and $v_i\overleftarrow{C}v_j$ denote, respectively, the paths $v_iv_{i+1}\ldots v_j$ and $v_iv_{i-1}\ldots v_j$ (subscripts expressed modulo $t$).  Let $C=v_0 v_1, \ldots v_{t-1} v_1$ be a non-extendable cycle in a graph $G$. With reference to a given non-extendable cycle $C$, a vertex of $G$ will be called a \emph{cycle vertex} if it is on $C$, and an \emph{off-cycle }vertex if it is in $V(G)-V(C)$. A cycle vertex that is adjacent to an off-cycle vertex will be called an \emph{attachment vertex}. The following useful result was established in \cite{AFOW13}.

\begin{lemma} \label{l1} Let $C=v_0v_1\ldots v_{t-1} v_0$ be a non-extendable cycle of length $t$ in a graph $G$. Suppose $v_i$ and $v_j$ are two distinct attachment vertices of $C$ that have a common off-cycle neighbour $x$. Then the following hold. (All subscripts are expressed modulo $t$.)

\begin{itemize}
\item[\emph{1}.] $j\neq i+1$ and $j \ne i-1$.
\item[\emph{2}.] $v_{i+1} \nsim v_{j+1}$ and $v_{i-1} \nsim v_{j-1}$.

\item[\emph{3}.] If $v_{i-1} \sim v_{i+1}$, then $v_{j-1} \nsim v_i$ and $v_{j+1} \nsim v_i$.
\item[\emph{4}.] If $j=i+2$, then $v_{i+1}$ does not have two adjacent neighbours $v_k,v_{k+1}$ on the path  $v_{i+2}\overrightarrow{C}v_i$.
\end {itemize}
\end{lemma}

\begin{figure}[h]
\begin{center}
\includegraphics*[width=5in]{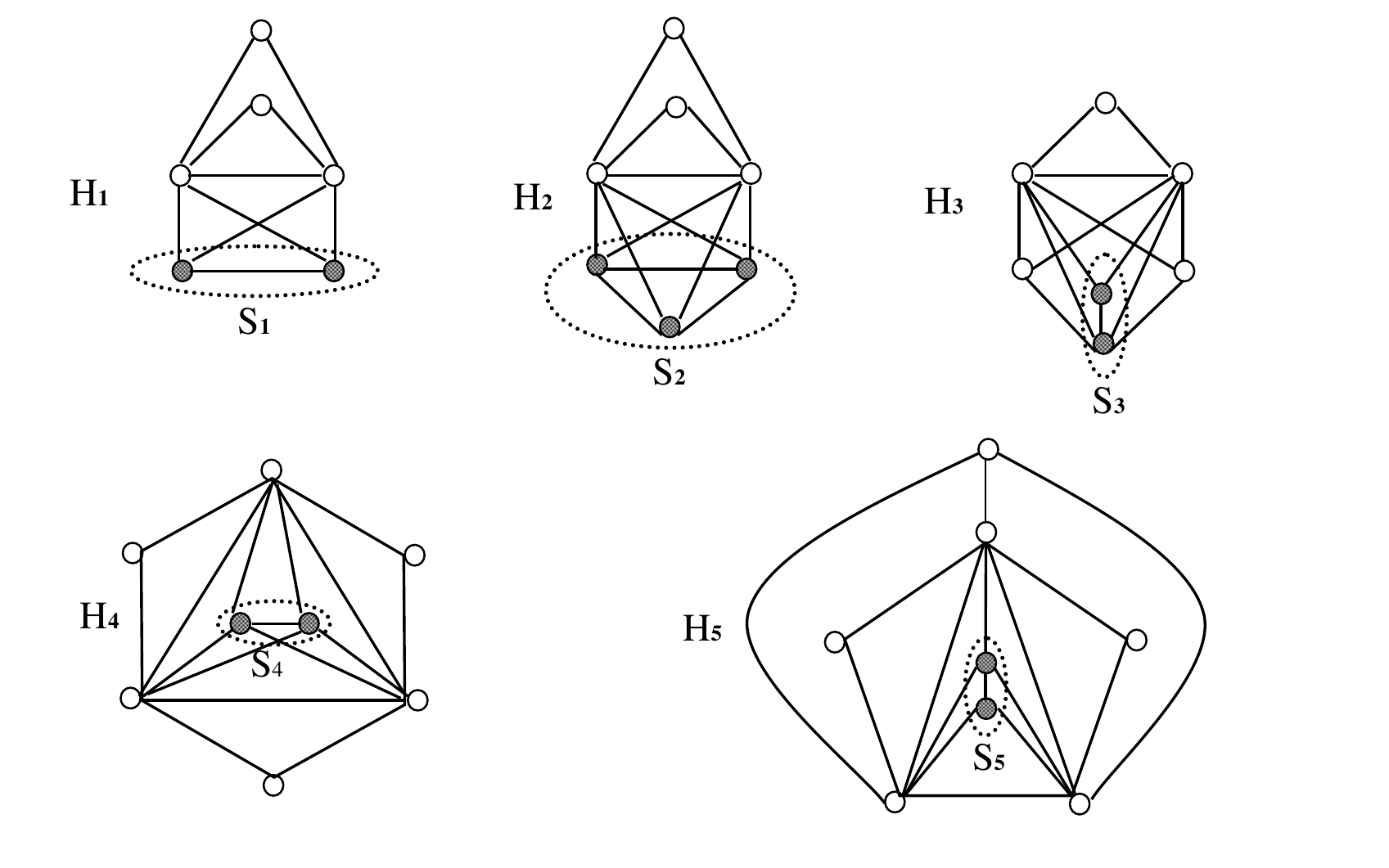}
\caption{Forbidden strong induced subgraphs $H_i$ with specified attachment sets $S_i$ for $1 \le i \le 5$}
\label{forbidden_strong_induced_subgraphs}
\end{center}
\end{figure}

The graphs of Figure \ref{forbidden_strong_induced_subgraphs}  show  forbidden strong induced subgraphs $H_i$ with specified attachment sets $S_i$, $1 \le i \le 5$.  Let $F_1$ be the graph $K_3 + \overline{K}_4$ and let $F_2$, $F_3$ and $F_4$ be the graphs shown in Figure \ref{forbidden_graphs}. The graphs $H_i$, $1 \le i \le 5$, and $F_i$, $1 \le i \le 4$, are referred to in Lemma \ref{l2} and Theorem \ref{t1}.

It is shown in Theorem \ref{t1} that a locally connected graph with minimum clustering coefficient $1/2$ and maximum degree 5 or 6 is fully cycle extendable if and only if it is not isomorphic to any of the graphs $F_i$, $1 \le i \le 4$, and does not contain any of the $H_j$s as strong induced subgraphs for $1 \le j \le 5$. Observe that if a connected graph cannot be isomorphic to a graph $F_i$, $1 \le i \le 4$, this is equivalent to saying that it cannot have $F_i$ as strong induced subgraph with empty attachment set.

\begin{figure}[h]
\begin{center}
\includegraphics*[width=4in]{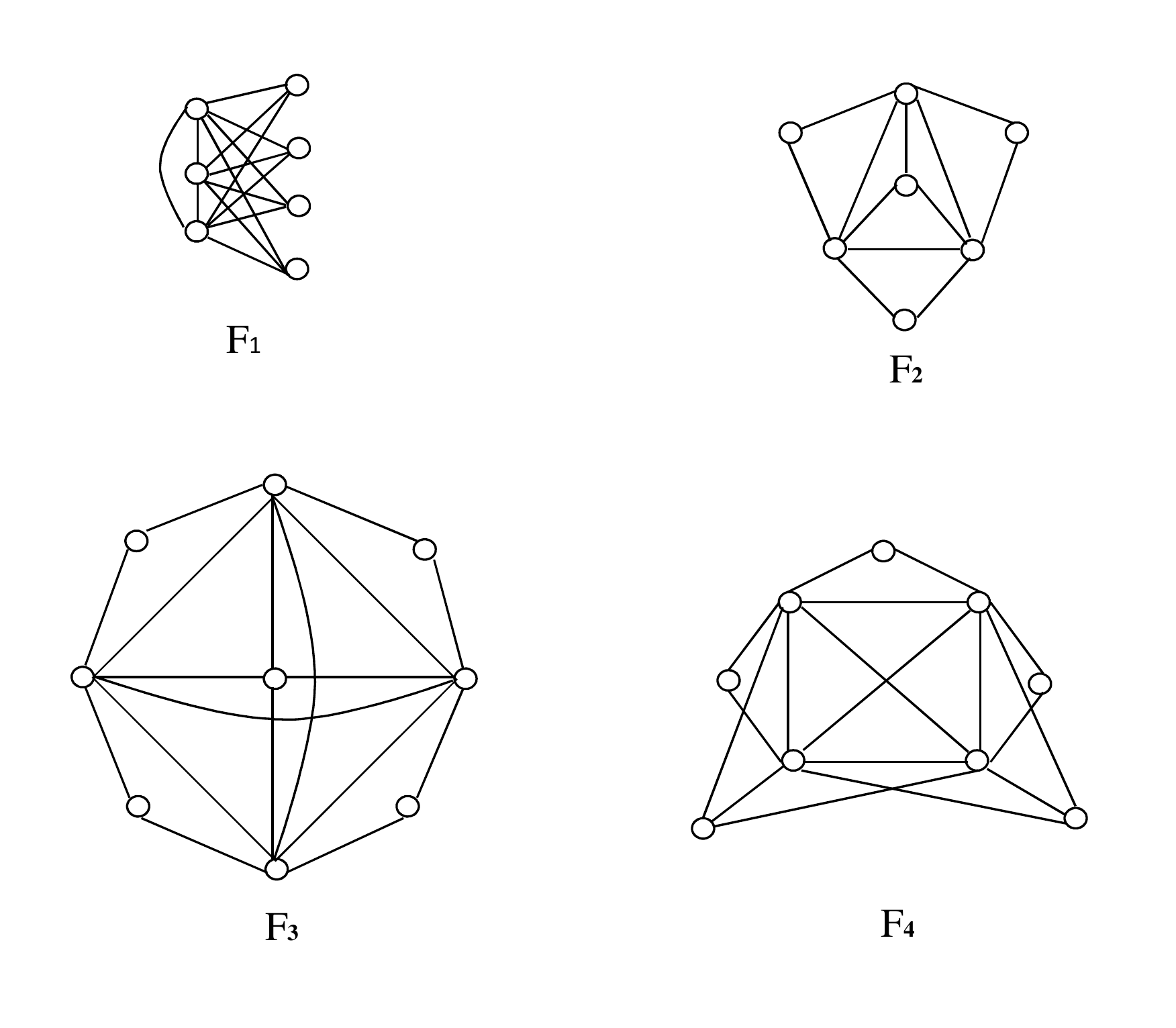}
\caption{Forbidden graphs}
\label{forbidden_graphs}
\end{center}
\end{figure}

\medskip

\begin{lemma}\label{l2}
Let $G$ be a connected locally connected graph with maximum degree $ \Delta \in\{5, 6\}$ and minimum clustering coefficient $\xi(G)\geq\frac{1}{2}$. Let $C=v_0v_1\ldots v_{t-1}v_0$ be a non-extendable, nonhamiltonian cycle in $G$ and let $v_0$ be an attachment vertex of largest possible degree. If $\deg(v_0)\geq 5$ and $v_0$ has a true twin, then $G$ contains $H_i$ as a strong induced subgraph with attachment sets $S_i$ for some $i $,  $1\le i \le 3$ as shown in Fig. \ref{forbidden_strong_induced_subgraphs}, or $G \cong F_1$.
\end{lemma}

\begin{proof}
First, assume $\deg(v_0)=5$. Since $\xi(v_0) \ge \frac{1}{2}$, it follows that $|E(\langle N(v_0) \rangle)| \ge 5$. Since $G$ is locally connected and no off-cycle neighbour of $v_0$ is adjacent with $v_1$ or $v_{t-1}$, by Lemma \ref{l1}(1), $v_0$ has either one or two off-cycle neighbours. Let $N(v_0)=\{x,v_1,v_i,v_j,v_{t-1}\}$ or $N(v_0)=\{x,y,v_1,v_i,v_{t-1}\}$, respectively, where $x$ and $y$ are off-cycle neighbours of $v_0$ and  $1<i<j<t-1$. By Lemma \ref{l1}(1), no off-cycle neighbour of $v_0$ can be a true twin of $v_0$. Assume $v_i$ is a true twin of $v_0$.

In the case where $v_0$ has one off-cycle neighbour, it follows that $i=2$ and $j=3$. By Lemmas \ref{l1}(1) and \ref{l1}(2), $x \not\sim\{v_1,v_{t-1},v_3\}$ and $v_{1} \not\sim \{v_{t-1}, v_3\}$. Since $\xi(v_0) \ge 1/2$, it follows that $v_3 \sim v_{t-1}$. Since $G$ is locally connected and $\deg(v_0)=\deg(v_2)=5$, it follows that $\deg(v_1) =\deg(x) =2$. Hence $G$ contains $H_1$ as a strong induced subgraph with attachment set $\{v_3,v_{t-1}\}$.

If $v_0$ has two off-cycle neighbours, $t=4$. By Lemmas \ref{l1}(1) and \ref{l1}(2), $\{x,y\} \not\sim \{v_1, v_{t-1}\}$ and $v_1 \nsim v_{t-1}$. Since $\xi(v_0) \ge 1/2$, we have $x \sim y$. As $\deg(v_0)=\deg(v_i)=5$ and $G$ is locally connected, $\deg(v_1)=\deg(v_{t-1})=2$. Hence $G$ has $H_1$ as strong induced subgraph with attachment set $\{x,y\}$.

Next, assume $\deg(v_0)=6$. Since $\xi(v_0)\geq\frac{1}{2}$, it follows, that $|E(\langle N(v_0) \rangle)| \ge 8$. It now follows from the assumption that $v_0$ has a true twin, and by Lemmas \ref{l1}(1), and \ref{l1}(2) that $v_0$ cannot have exactly two off-cycle neighbours; otherwise $\xi(G)< \frac{1}{2}$. Thus $v_0$ has either exactly one or exactly three off-cycle neighbours. Let $N(v_0)=\{x,v_1,v_i,v_j,v_k,v_{t-1}\}$ or $N(v_0)=\{x,y,z,v_1,v_i,v_{t-1}\}$, respectively, where $x,y$ and $z$ are off-cycle neighbours and $1<i<j<k<t-1$. If, in the first case, $v_j$ is a true twin of $v_0$, it follows from Lemmas \ref{l1}(1), \ref{l1}(2), and \ref{l1}(3), and the fact that $\Delta=6$ that $\xi(v_0)<\frac{1}{2}$. Hence, we may assume, in either case, that $v_i$ is a true twin of $v_0$.  If $v_0$ has exactly three off-cycle neighbours, then it follows since $G$ is locally connected and the facts that $\Delta=6$ and $\xi(v_0) \ge 1/2$ and by Lemmas \ref{l1}(1) and \ref{l1}(2), that $t=4$. Thus $G$ contains $H_2$ as a strong induced subgraph with attachment set $\{x,y,z\}$. Suppose now that $v_0$ has exactly one off-cycle neighbour. Since $\Delta =6$ and $v_0$ and $v_i$ are true twins, $i=2$ and $j=3$. By Lemmas \ref{l1}(1), \ref{l1}(2), $x \not\sim \{v_1,v_3,v_{t-1}\}$ and $v_1 \not\sim \{v_3,v_{t-1}\}$. We now consider four cases depending on which subset of $\{v_1,x\}$ the vertex $v_k$ is adjacent with.\\

\noindent{\bf Case 1}  $v_k \not\sim \{v_1,x\}$. Since $\xi(v_0) \ge 1/2$ it follows that in this case, $v_k \sim \{v_3,v_{t-1}\}$ and $v_3 \sim v_{t-1}$.  Since $\Delta = 6$, $v_0$ and $v_2$ do not have neighbours other than those already mentioned. Moreover, since $x$ and $v_i$ are adjacent with both $v_0$ and $v_{t-1}$, and  $G$ is locally connected, neither $x$ nor $v_1$ can have any other neighbours. Hence the subgraph induced by $\{v_0,v_1,v_2,v_3,v_k,v_{t-1},x\}$ is a strong induced subgraph isomorphic to $H_2$ and with attachment set $\{v_3,v_k,v_{t-1}\}$.\\

\noindent{\bf Case 2} $v_k \sim v_1$ and $v_k \not\sim x$. Then at least two of the three edges $v_kv_{t-1}, v_kv_3,$ $ v_3v_{t-1}$ belong to $G$. Suppose first that $v_k \sim \{v_3,v_{t-1}\}$. Since $\Delta = 6$, either $k=t-2$ or $k=j+1=4$. We may assume the former; the second case can be argued similarly. If $v_3 \sim v_{t-1}$, then $v_{t-1}v_3 \overrightarrow{C} v_kv_1v_0xv_2v_{t-1}$ is an extension of $C$. So $v_3 \not\sim v_{t-1}$.  Since $\Delta =6$ and  $G$ is locally connected $\deg(x) =2$. Also $\deg(v_1) = 3$, for if $v_1$ has a neighbour $y \not\in \{v_0,v_2,v_k\}$, then $y \sim v_k$.  Note, by Lemma \ref{l1}(2), $y \not\in \{v_3, v_{t-1}\}$. Thus $y$ is an off-cycle vertex. Since $\deg(v_k) \le 6$, $t=6$ and $k=4$ and $v_1yv_4v_3v_2v_5v_0v_1$ is an extension of $C$, which is not possible. We now show that $\deg(v_{t-1}) =3$. Suppose $v_{t-1}$ has a neighbour $y$ other than $v_0, v_2$ and $v_k$. Since $G$ is locally connected and $\Delta = 6$, $y \sim \{v_{t-1}, v_k\}$. By Lemma \ref{l1}(1), $y$ is not an off-cycle vertex. Thus $y = v_{k-1}$ (and by assumption $k-1 \ne 3$). But now $v_{t-1}v_{k-1} \overleftarrow{C}v_2xv_0v_1v_kv_{t-1}$ is an extension of $C$. Hence $\deg(v_{t-1})=3$. So $G$ contains $H_3$ as strong induced subgraph with attachment set $\{v_k, v_3\}$.

We now assume that $v_3 \sim v_{t-1}$ and, by the above observation, $v_k$ is adjacent with exactly one of $v_3$ or $v_{t-1}$, say $v_k \sim v_3$. The case where $v_k \sim v_{t-1}$ and $v_k \nsim v_3$ can be argued similarly. Since $v_k \nsim v_{t-1}$ it follows that $k \ne t-2$. If $v_{k+1} \sim v_3$, then $C$ can be extended to the cycle $v_{k+1}v_3\overrightarrow{C}v_kv_1v_2xv_0v_{t-1}\overleftarrow{C}v_{k+1}$. So $v_{k+1} \not\sim v_3$. Since $\Delta =6$ and by Lemma \ref{l1}(4) $v_{k+1} \not\sim \{v_0,v_1,v_2\}$.  By Lemma \ref{l1}(1) it follows that $v_k$ and $v_{k+1}$ do not have a common off-cycle neighbour. Since $\langle N(v_k) \rangle $ is connected it now follows that $v_{k+1} \sim v_{k-1}$ and hence that $v_{k-1} \ne v_3$.  But then $v_0xv_2v_1v_kv_3\overrightarrow{C}v_{k-1}v_{k+1}\overrightarrow{C}v_0$ is an extension of $C$ which is not possible.\\

\noindent{\bf Case 3} $v_k \sim x$ and $v_k \not\sim v_1$. Then at least two of the three edges $v_kv_{t-1}, v_kv_3,$ $v_3v_{t-1}$ belong to $G$. Suppose first that $v_k \sim \{v_3,v_{t-1} \}$. Then either $k=t-2$ or $k=j+1=4$. We consider the case where $k=t-2$. By the symmetry of the structure, the case where $k=j+1=4$ can be argued similarly. Since $k=t-2$, it follows from Lemma \ref{l1}(4) that $v_3 \not\sim v_{t-1}$.  If $k-1 \ne 3$, then it follows since $\langle N(v_k) \rangle $,  is connected and  $\Delta = 6$ and by Lemmas \ref{l1}(1) and \ref{l1}(2), that $v_{k-1} \not\sim \{x,v_0,v_2, v_{t-1}\}$ and thus $v_{k-1} \sim v_3$. Since $G$ is locally connected  it can now be shown that $\deg(v_1)=2$ and $\deg(v_{t-1})=\deg(x) =3$. So the subgraph induced by $\{x,v_0,v_1,v_2,v_3,v_k,v_{t-1}\}$ is a strong induced subgraph isomorphic to $H_3$ with attachment set $\{v_3,v_k\}$. Suppose now that $k-1=3$. So $t=6$. Since $v_4$ does not have a common off-cycle neighbour with either $v_3$ or $v_5$, $\deg(v_3)=\deg(v_5) =3$. Also $\deg(v_1) =2$ since $G$ is locally connected and $v_1 \not\sim \{v_3,v_4,v_5, x\}$. Hence $G$ contains the graph $H_3$ as strong induced subgraph with attachment set $\{v_4,x\}$.

Assume next that $v_k$ is not adjacent with both $v_3$ and $v_{t-1}$. Then $v_3 \sim v_{t-1}$ and $v_k$ is adjacent with exactly one of $v_3$ and $v_{t-1}$. We consider the case where $v_k \sim v_{t-1}$; the case where $v_k \sim v_3$ can be argued similarly. By Lemma \ref{l1}(4) $k \ne t-2$. By Lemma \ref{l1}(1) and the fact that $\Delta =6$, $\{v_{k-1}, v_{k+1}\} \nsim \{x,v_0,v_2\}$, and by Lemmas \ref{l1}(1) and \ref{l1}(2), $v_{k-1} \nsim \{v_{t-1}, v_{k+1}\}$. So $\xi(v_k) < 1/2$ which is not possible.\\

\noindent{\bf Case 4} $v_k \sim \{x, v_1\}$. Suppose first that $k=t-2$ or $k=j+1=4$. We consider the case where $k=t-2$; the case where $k=j+1=4$ can be argued similarly.  Then, by Lemma \ref{l1}(4), $v_3 \nsim v_{t-1}$. Suppose $k \ne 4$. Since $\Delta = 6$ and by Lemmas \ref{l1}(1) and \ref{l1}(2), $v_{k-1} \not\sim \{x,v_0,v_1,v_2,v_{t-1}\}$. Since $\langle N(v_k) \rangle $ is connected, this is not possible. Hence $k=4$. In this case it can be argued, using Lemma \ref{l1} and the fact that $\Delta =6$, that $\deg(v_1)=\deg(v_3)=\deg(v_5)=\deg(x)=3$. Thus $G$ is isomorphic to the graph $F_1$ which is not possible. Assume now that $k \not\in \{4, t-2\}$, Then by Lemmas \ref{l1}(1), \ref{l1}(4) and the fact that $\Delta=6$, we see that $\{v_{k-1}, v_{k+1}\} \nsim \{x, v_0, v_1, v_1\}$. So $\xi(v_k) < 1/2$, a contradiction.

\end{proof}

\begin{theorem}
\label{t1}
Let $G$ be a locally connected graph with maximum degree $\Delta=5$ or $6$ and minimum clustering coefficient $\xi(G)\geq\frac{1}{2}$. Then $G$ is fully cycle extendable if and only if $G$ is not isomorphic to $F_i$, $1 \le i \le 4$, and  $G$ does not contain the graphs $H_i$,  with attachment sets $S_i$, for $1 \le i \le 5$, as strong induced subgraphs.
\end{theorem}

\begin{proof}
It is readily seen that if $G$ is one of the graphs $F_i$, $1 \le i \le 4$, then $G$ is not hamiltonian and hence not fully cycle extendable. To see this observe that each of these graphs contains a vertex cutset $U$ such that the number of components of $G-U$ exceeds $|U|$. For $F_1$, $F_3$, and $F_4$ the set $U$ consisting of the vertices of degree $6$ is such a cutset and for $F_2$ the set $U$ consisting of the vertices of degree 5 is such a vertex cutset. It follows that $G$ is not hamiltonian see \cite{bm, CLZ}. If $G$ contains $H_i$ as a strong induced subgraph where $i \in \{1,2,3,4,5\}$ with the  attachment sets as shown in Figure \ref{forbidden_strong_induced_subgraphs}, then we can again see that $G$ contains a vertex cutset $U$ such that the number of components of $G-U$ exceeds $|U|$, and hence $G$ is not hamiltonian. For $H_1$, let $U$ consist of the vertices having degree 5 in $H_1$, for $H_2$ let $U$ consist of the vertices having degree 6 in $H_2$, for $H_3$ let $U$ consist of the vertices having degree 6 and the vertex having degree 5 in $H_3$, for $H_4$ let $U$ consist of the three vertices having degree $6$ in $H_4$ and for $H_5$ let $U$ consist of the three vertices of degree 6 in $H_5$.

For the converse, assume that $G$ is not isomorphic to one of the graphs $F_i$ and that $G$ does not contain the graphs $H_i$ with attachement set $S_i$ as induced subgraphs, $1 \le i \le 5$.  Suppose, to the contrary, that $C = v_0v_1\ldots v_{t-1}v_0$ is a non-extendable cycle in $G$ for some $t < n$ (indices taken mod $t$). Since $t<n$ and $G$ is connected, there must be a vertex of $C$ that is an attachment vertex. We may assume that $v_0$ is an attachment vertex on $C$ of maximum degree. Let $x$ be an off-cycle neighbour of $v_0$. Since $G$ is locally connected and $x \nsim \{v_1, v_{t-1}\}$ it follows that $\deg(v_0) \ge 4$. Since  $\Delta = 6$, $\deg(v_0) \leq 6$.\\

\noindent\textbf{Case 1} Suppose $\deg(v_0) \leq 5$. Suppose first that $\deg(v_0) =4$ and let $N(v_0) = \{x, v_1, v_i, v_{t-1}\}$. Since $\xi(v_0) \geq \frac{1}{2}$, $\langle N(v_0)\rangle$ must contain three edges. By Lemma \ref{l1}(1) $x\nsim \{v_{t-1}, v_1\}$, so $v_i \sim x$ and $v_i$ is adjacent to at least one of $v_1$ and $v_{t-1}$, say $v_i\sim v_1$. If $v_{t-1}\sim v_1$, then by Lemma \ref{l1}(2), $i-1\neq 1$ and $i+1\neq t-1$. But then $\deg(v_i)\geq 5$, contradicting the fact that $v_0$ is an attachment vertex of maximum degree. So $v_{t-1}\nsim v_1$. Therefore $v_i\sim \{x, v_1,v_{t-1}\}$. Since $v_i$ has an off-cycle neighbour $x$, $\deg(v_i) \leq \deg(v_0)$ by our choice of $\deg(v_0)$. Hence $\deg(v_i) = 4$ and so $i-1=1$ and $i+1 = t-1$. Therefore $t=4$. By Lemmas \ref{l1}(1) and \ref{l1}(2), there are no additional edges joining pairs of vertices of $V(C) \cup \{x\}$. Suppose $G \neq \langle V(C) \cup \{x\}\rangle$. Then there exists a vertex $y \notin \langle V(C) \cup \{x\}\rangle$ adjacent with a vertex of $C$. Since G is a connected graph, $y$ is adjacent to at least one vertex of $V(C) \cup \{x\}$. By our choice of $v_0$, it follows that  $y\nsim \{v_0, v_2\}$, so $y$ is adjacent to at least one of $x, v_1, v_{t-1}$. If $y$ is adjacent with $v_1$ or $v_{t-1}$, say $v_1$, then it follows, since $\langle N(v_1) \rangle$ is connected, that some off-cycle neighbour of $v_1$ must also be adjacent with $v_0$ or $v_2$, contrary to Lemma \ref{l1}(1). So $y \not\sim \{v_1,v_{t-1}\}$. Thus $y \sim x$. Since $G$ is locally connected, $G$ has no cut-vertices. So some off-cycle vertex is adjacent with $x$ and at least one of $v_0$ and $v_2$. This contradicts our choice of $v_0$.

We next assume that $\deg(v_0) =5$. By our hypothesis and Lemma \ref{l2}, $v_0$ does not have a true twin. Since $G$ is locally connected, $G$ has at most two off-cycle neighbours. Suppose first that $v_0$ has two off-cycle neighbours $x$ and $y$. Let $N(v_0) =\{x,y,v_1,v_i,v_{t-1}\}$.  By Lemma \ref{l1}(1) $\{x,y\} \not\sim \{v_1,v_{t-1}\}$. Since $G$ is locally connected, $v_i$ must be adjacent with at least one of $x$ and $y$ and at least one of $v_1$ and $v_{t-1}$. Suppose first that $v_i \sim \{x,y\}$. If $v_i \sim \{v_1,v_{t-1}\}$, then $v_0$ and $v_i$ are true twins which is not possible. We may thus assume that $v_i \sim v_1$ and $v_i \not\sim v_{t-1}$. By our choice of $v_0$, $i=2$. Since $\xi(v_0) \ge 1/2$,  $v_1 \sim v_{t-1}$. By Lemma \ref{l1}(2) the latter is not possible.
Thus $v_i$ is adjacent with exactly one of $x$ or $y$, say $x$. By Lemma \ref{l1}(1) $\{x,y\} \not\sim \{v_1,v_{t-1}\}$ and by assumption $y \nsim v_i$. Since $\xi(v_0) \ge 1/2$, $v_i \sim \{v_0,v_1,v_{t-1}\}$, $x \sim y$ and $v_1 \sim v_{t-1}$. By Lemma \ref{l1}(2), $i \ne 2$ and $i \ne t-2$. So $\deg(v_i) \ge 6$, contrary to our choice of $v_0$. So this case cannot occur.

Thus all attachment vertices have exactly one off-cycle neighbour. In particular $v_0$ has exactly one off-cycle neighbour $x$ and four cycle neighbours $\{v_1,v_i,v_j,v_{t-1}\}$ where $i < j$. By Lemma \ref{l1}(1), $x \not\sim \{v_1,v_{t-1}\}$.  Since $G$ is locally connected, $x$ is adjacent with at least one of the two vertices $v_i$ and $v_j$, say $v_i$.

Suppose first that $v_1 \sim v_{t-1} $. By Lemma \ref{l1}(2), $i-1 \ne 1$. By Lemmas \ref{l1}(1) and \ref{l1}(2), $\{x, v_0\} \not\sim \{v_{i-1}, v_{i+1}\}$. Since $\xi(v_i) \ge 1/2$, we see that $\deg (v_i) \ge 5$. From our choice of $v_0$, $\deg (v_i) =5$.  So $v_i$ has a fifth neighbour $v_q$. Suppose $q=t-1$. By our choice of $v_0$ and by Lemma \ref{l1}(1), \ref{l1}(2) and \ref{l1}(3), $v_{i-1} \not\sim \{x, v_0, v_{i+1},v_{t-1}\}$. This contradicts the fact the $\langle N(v_i) \rangle$ is connected. So $q \ne t-1$. Similarly $q \ne 1$.

If $v_i \sim x$, then we can argue as for $v_i$ that $v_j \nsim \{v_1, v_{t-1}\}$. However, then there are six non-adjacent pairs in $\langle N(v_0) \rangle$, contrary to the fact that $\xi(v_0) \ge \frac{1}{2}$. So $v_j \nsim x$. Since $\xi(v_0) \ge \frac{1}{2}$, it follows that $v_j \sim \{v_1,v_i, v_{t-1}\}$. By Lemma \ref{l1}(2), $v_{i+1} \nsim v_1$. So $j \ne i+1$. Suppose $j =t-2$. Then $v_{i-1} \nsim v_{t-2}$; otherwise, $v_{t-2}v_{i-1} \overleftarrow{C} v_1v_{t-1}v_0xv_i \overrightarrow{C}v_{t-2}$ is an extension of $C$, a contradiction.  Since $x \nsim \{v_{i-1}, v_{i+1}, v_{t-2}\}$ and $v_0 \nsim \{v_{i-1}, v_{i+1}\}$ and $\xi(v_i) \ge \frac{1}{2}$, it follows that $v_{i-1} \sim v_{i+1}$. Thus $v_{t-2} \overleftarrow{C}v_{i+1}v_{i-1} \overleftarrow{C}v_1v_{t-1}v_0xv_iv_{t-2}$ is an extension of $C$, which is not possible. So $j \not\in \{i+1, t-2\}$. Since $\Delta \le 6$ we see that $\deg (v_i) =6$. Hence $v_{i-1} \nsim v_j$. We also know that $\{x, v_0\} \nsim \{v_{i-1}, v_{i+1}\}$. Since $\xi(v_i) \ge \frac{1}{2}$, it now follows that $v_{i-1} \sim v_{i+1}$ and $v_{i+1} \sim v_j$. Since $\deg (v_j)=6$, it follows that $j=i+2$. So $v_{i+2} \overrightarrow{C}v_{t-1}v_1 \overrightarrow{C}v_{i-1}v_{i+1}v_ixv_0v_{i+2}$ is an extension of $C$, a contradiction.
We conclude that $v_1 \not\sim v_{t-1}$.

Next suppose that $v_i \sim v_{t-1}$. Then, by our choice of $v_0$,   $N(v_i) =\{x,v_0,v_{i-1},$ $v_{i+1},v_{t-1}\}$. If $i-1 \ne 1$, then $v_{i-1} \nsim v_0$. By Lemma \ref{l1}(1), \ref{l1}(2) and \ref{l1}(3), $x \not\sim \{v_{t-1},v_{i-1},$ $v_{i+1}\}$, $v_{i-1} \nsim v_{t-1}$ and $v_{i-1} \nsim v_{i+1}$, contrary to the fact that $\xi(v_i) \ge 1/2$. So $i-1 = 1$. Since $\xi(v_i) \ge 1/2$, it follows that $v_{i+1} \sim \{v_0,v_{t-1}\}$. Thus $i+1=j$. Hence $v_i$ is a universal vertex in $N(v_0)$. By Lemma \ref{l2}, this is not possible. We conclude that
\begin{center}
$v_i\sim x$ implies $v_i\nsim v_{t-1}$, and similarly $v_j\sim x$ implies $v_j\nsim v_1$.
\end{center}

\noindent Suppose $v_j \sim x$. Since $\xi(v_0) \ge 1/2$, $v_i \sim \{v_1,v_j\}$ and $v_j \sim \{x, v_{t-1}\}$. By our choice of $v_0$, $i=2$ and $j=t-2$. If $j = i+1$, then $C$ is extendable. Hence $j \ne i+1$. By our choice of $v_0$ and Lemmas \ref{l1}(1) and \ref{l1}(2), $v_{i+1} \nsim\{x,v_0,v_1\}$. Since $\langle N(v_i) \rangle$ is connected $v_{i+1} \sim v_j$. By our choice of $v_0$, $j = i+2 =4$. So $t=6$. Since $G$ is locally connected and from our choice of $v_0$, no vertex of $\{x,v_0, \ldots, v_5\}$ is adjacent with a vertex not in this set. Hence $G$ is isomorphic to the forbidden graph $F_2$ which is not possible.
 Hence, if $v_i \sim x$, then $v_{t-1} \nsim \{v_{1},v_i \}$ and $v_j \nsim x$. Moreover, by Lemma \ref{l1}(1), $x \nsim \{v_1,v_{t-1}\}$. Since $\xi(v_0) \ge 1/2$ it follows that $v_j \sim \{v_1,v_i,v_{t-1}\}$ and $v_i \sim v_1$.  By Lemma \ref{l1}(2), $i+1 \ne j$. So $i=2$; otherwise, $\deg(v_i) >5$. We show next that either $j-1=i+1$ or $j+1 = t-1$; otherwise, $\xi(v_j) < 1/2$. Suppose $j-1 \ne i+1$ and $j+1 \ne t-1$. Since $\deg(v_0) =\deg(v_i) =5$ and by Lemma \ref{l1}(3), $\{v_{j-1}, v_{j+1}\} \nsim \{v_0,v_1,v_i\}$. Moreover, by the case we are in, $v_1 \nsim v_{t-1}$. Also $v_{j-1} \nsim v_{t-1}$; otherwise,  $v_0xv_2 \overrightarrow{C}v_{j-1}v_{t-1} \overleftarrow{C} v_jv_1v_0$ is an extension of $C$. Thus $\xi(v_j) < 1/2$, which is not possible.

Suppose now that $j-1=i+1=3$.  Suppose first that $t-1 \ne j+1(=5)$. We consider $N(v_4)=\{v_0,v_1, v_2, v_3, v_{t-1}, v_{5}\}$.
Since $\deg(v_0)=\deg(v_2) =5$ and by Lemma \ref{l1}(3), $v_{5} \nsim\{v_0, v_1, v_2\}$ also $v_{5} \nsim v_{3}$; otherwise, $v_0xv_2v_3v_5 \overrightarrow{C}v_{t-1}v_4v_1v_0$ is an extension of $C$. Since $\deg(v_0) = 5$ and by Lemma \ref{l1}(2) $v_{3} \nsim \{v_0,v_1\}$. By Lemma \ref{l1}(2) and the above, $v_{t-1} \nsim \{v_1, v_2\}$. So $\xi(v_4) < 1/2$.

Suppose next that $j-1=i+1=3$ and $t-1 = j+1 =5$. As observed, $x \nsim\{v_1,v_4,v_{t-1}\}$, $v_1 \nsim v_{t-1}$ and $v_2 \nsim v_{t-1}$. Thus, since $\xi(v_0) \ge 1/2$ we see that $v_j(=v_4) \sim\{v_1, v_2\}$. By Lemma \ref{l1}(2), $v_1 \nsim\{v_3, v_5\}$ and from the case we are in $v_2 \nsim v_5$ and $v_0 \nsim v_3$. Also $v_3 \nsim v_5$; otherwise, $v_0xv_2v_1v_4v_3v_5v_0$ is an extension of $C$. Since $v_4$ has five neighbours on $C$ and by our choice of $v_0$, vertex $v_4$ cannot have an off-cycle neighbour. So $\deg(v_4) = 5$. Since $\deg(v_0)=\deg(v_2)=\deg(v_4)=5$ and  $G$ is locally connected, $\deg(x)=\deg(v_3)=\deg(v_5)=2$ and $\deg(v_1) =3$. So $G$ is isomorphic to the graph $F_2$ which is not possible.

Suppose $j+1=t-1$ and $j-1 \ne i+1$. Since $\xi(v_0) \ge 1/2$, we can argue as before that $v_j \sim \{v_0,v_1, v_2, v_{j-1}, v_{t-1}\}$. By our choice of $v_0$, vertex $v_j$ has no off-cycle neighbours. By Lemma \ref{l1} and the case we are in we see that in $\langle N(v_2) \rangle$, $x \nsim \{v_1, v_{3}, v_j\}$, and $v_{3} \nsim \{v_0, v_1\}$. Since $\deg(v_2) = 5$, and $\xi(v_2) \ge 1/2$ we have $v_{3} \sim v_j$. Now $v_{j-1} \nsim v_{j+1}(=v_{t-1})$; otherwise, $v_0xv_2v_1v_jv_3 \overrightarrow{C}v_{j-1}v_{j+1}v_0$ is an extension of $C$. Using this observation, Lemma \ref{l1}, and the case we are in we see that in $\langle N(v_j) \rangle$, $v_{j-1} \nsim \{v_0, v_1, v_2, v_{t-1}\}$, $v_1 \nsim \{v_3, v_{t-1}\}$, $v_2 \nsim v_{t-1}$ and $v_0 \nsim v_3$. Hence $\xi(v_j) < 1/2$, a contradiction.  Hence this case cannot occur.\\

\noindent\textbf{Case 2} Assume that $\deg(v_0)=6$. Since $\xi(v_0) \ge 1/2$, it follows that $|E(\langle N(v_0) \rangle)| \ge 8$. Hence $\langle N(v_0) \rangle$ has at most seven non-adjacencies.\\

\noindent{\bf Subcase 2.1} Assume $v_0$ has three off-cycle neighbours $x,y$ and $z$. Then $v_0$ has exactly one cycle neighbour, $v_i$ say, distinct from $v_1$ and $v_{t-1}$.  By Lemma \ref{l1}(1), $\{v_1, v_{t-1}\} \nsim \{x,y,z\}$ and by Lemma \ref{l2}, $v_i$ is not a true-twin of $v_0$. Since $\xi(v_0) \ge 1/2$ it follows that $v_0$ is non-adjacent with exactly one neighbour of $v_0$ and $v_1 \sim v_{t-1}$. Since $\Delta = 6$ either $i=2$ or $i=t-2$. In either case it follows from Lemma \ref{l1}(2), that $C$ is extendable.  So this case cannot occur.\\

\noindent\textbf{Subcase 2.2} Assume $v_0$ has two off-cycle neighbours $x$ and $y$. Let $N(v_0)=\{x,y,v_1,v_i,$ $v_j,v_{t-1}\}$ where $1<i<j<t-1$. Since $\langle N(v_0)\rangle$ is connected, and by Lemma \ref{l1}(1), we may assume that $x\sim v_i$. Suppose that $v_i\sim v_1$. Assume first that $i-1\ne1$ and $i+1\ne j$.  Since $\deg(v_0)=6$,  it follows that $v_0, \nsim \{v_{i-1},v_{i+1} \}$. By Lemmas \ref{l1}(1), \ref{l1}(2) and \ref{l1}(3), $x \nsim \{v_1, v_{i-1}, v_{i+1}\}$ and $v_{i+1} \nsim \{v_1,v_{i-1}\}$. Since $\xi(v_i) \ge 1/2$ and $\Delta =6$ we have $\deg(v_i)=6$. So $v_i$ has a neighbour not in $\{x,v_0,v_1,v_{i-1}, v_{i+1}\}$. Since $\xi(v_i) \ge 1/2$, this  implies that such a neighbour is a true twin of $v_i$.  By Lemma \ref{l2} and our hypothesis, this is not possible.  Hence $i-1=1$ or $i+1=j$.

Assume $i+1=j$.  By Lemmas \ref{l1}(1) \ref{l1}(2), \ref{l1}(3), $\{x,y\} \nsim \{v_1,v_{t-1} \}$, $x \nsim v_j$, $v_j \nsim v_1$ and $v_1 \nsim v_{t-1}$.  Since $\xi(v_0) \ge 1/2$, it follows that $v_i$ is a true twin of $v_0$. By Lemma \ref{l2}, this is not possible.  We conclude that $i+1\ne j$ and $i-1=1$, i.e. $i=2$. So $v_0 \nsim v_{i+1}$.

Next, assume that $v_2\sim v_{t-1}$. By Lemmas \ref{l1}(1), \ref{l1}(2), and the above $x \nsim \{v_1,v_{3}, v_{t-1} \}$, $v_1 \nsim\{v_{2}, v_{t-1}\}$ and $v_0 \nsim v_{3}$.  Since $\xi(v_2) \ge 1/2$ and  $\Delta = 6$,  it follows that $\deg(v_2)=6$. The neighbour of $v_2$ not in $\{x,v_0,v_1,v_{3},v_{t-1}\}$ must be adjacent with at least four of the five vertices in this set. Hence such a neighbour must be on $C$. Let $v_q$ be such a neighbour of $v_2$. Since $v_2$ and $v_q$ are not twins, and $\xi(v_2)\geq\frac{1}{2}$, it follows that $v_{3}\sim v_{t-1}$ and $v_q$ has exactly one non-neighbour in $\langle N(v_2)\rangle$.

Suppose $q \ne j$. Thus $v_q \nsim v_0$. Consequently, $v_q \sim \{x,v_{t-1},v_1,v_{3}\}$. Since $\Delta=6$, either $q=t-2$ or $q=4$. Suppose $q=t-2$. By applying Lemmas \ref{l1}(1) and \ref{l1}(2) and using the fact that $\Delta = 6$ one can show that, $\xi(v_q)<\frac{1}{2}$. If $q=i+2$, then it can be argued that $\xi(v_2) < 1/2$.

Hence, $q=j$. Suppose $x\sim v_j(=v_q)$. If $j-1=i+1=3$, then by Lemmas \ref{l1}(1) and \ref{l1}(2) we see that $v_{3} \nsim\{x, v_0, v_1,v_{t-1}\}$, $v_{i-1}(=v_1) \nsim\{x,v_{t-1}\}$ and $x \nsim v_{t-1}$. By Lemma \ref{l2}, $v_4$ is not a true twin of $v_2$. Hence either $v_4 \nsim v_1$ or $v_4 \nsim v_{t-1}$. So $\xi(v_2) < 1/2$, a contradiction. So $j-1 \ne i+1$. Similarly $j+1 \ne t-1$. This implies that $\\deg(v_j)>6$, a contradiction. We conclude that $x\nsim v_j$, and moreover that $v_j\sim \{v_{t-1},v_0,v_1,v_i,v_{i+1}\}$. So either $j-1=i+1$ or $j+1=t-1$. Suppose $j =i+2=4$. From the case we are in and by Lemmas \ref{l1}(1) and \ref{l1}(2) we see that $x \nsim \{v_1, v_{3}, v_4, v_{t-1} \}$, $v_1 \nsim \{v_{3}, v_{t-1} \}$, and $v_0 \nsim v_{3}$. Since $\xi(v_2) \ge 1/2$, $v_{3} \sim v_{t-1}$. So, by Lemma \ref{l1}(2), $v_4 \nsim y$. Using this observation, Lemmas \ref {l1}(1) and \ref{l2} and the fact that $\Delta = 6$, we see that $\{x,y\} \nsim \{v_1,v_4, v_{t-1}\}$, $y \nsim v_2$ and $v_1 \nsim v_{t-1}$; contrary to the fact that $\xi(v_0) \ge 1/2$. Suppose $j =t-2$. Using Lemmas \ref{l1}(1) and \ref{l1}(2) and the above we see that $x \nsim \{v_1, v_{3}, v_{t-2}, v_{t-1}\}$, $v_1 \nsim \{v_{3}, v_{t-1}\}$ and $v_{3} \nsim  v_0$. Since $\xi(v_2) \ge 1/2$, it follows that $v_{3} \sim \{v_{t-1}, v_{t-2}\}$. So $v_{t-2}v_1v_0xv_2v_{t-1}v_{3} \overrightarrow{C}v_{t-2} $ is an extension of $C$, contrary to assumption. We conclude that $v_2 \nsim v_{t-1}$.

It follows from a similar argument that we cannot have $v_j\sim \{x,v_{t-1},v_1\}$.  As before we see that $\{x,y\} \nsim \{v_1, v_{t-1}\}$ and $v_{t-1} \nsim \{v_1,v_i\}$.  Since $\xi(v_0) \ge 1/2$ it follows that $v_j$ is adjacent with at least one of $x$ or $y$. Suppose $v_j \sim v_{t-1}$. We can argue as for $v_i$, that $j=t-2$ and that $v_j \nsim v_1$. So $\{x,y\} \nsim \{v_1, v_{t-1}\}$, $v_1 \nsim \{v_j, v_{t-1}\}$ and $v_2 \nsim v_{t-1}$. Since $\xi(v_0) \ge 1/2$ we have $\{x,y\} \sim \{v_2, v_{t-2}\}$ and $v_2 \sim v_{t-2}$. By Lemma \ref{l1}(1), $j \ne i+1=3$. Since $\Delta =6$ and by Lemmas \ref{l1}(1) and \ref {l1}(2),  we have $v_{3} \nsim\{x,y,v_0,v_1, v_{t-2}\}$. Thus $\langle N(v_i) \rangle$ is not connected, a contradiction. So $v_j \nsim v_{t-1}$. Since $v_{t-1} \nsim \{x,y,v_1,v_i,v_j\}$, $\langle N(v_0) \rangle$ is not connected, so this case cannot occur.

Hence $v_i \nsim v_1$ and $v_j \nsim v_{t-1}$. By Lemma \ref{l1}(1) we also know that $\{x,y\} \nsim \{v_1,v_{t-1} \}$. This accounts for six non-edges in $\langle N(v_0) \rangle$. Since $\xi(v_0) \ge 1/2$, there is thus at most one additional non-adjacency in $\langle N(v_0) \rangle$. So either $v_i \sim \{x,y,v_{t-1}\}$ or $v_j \sim \{x,y,v_1\}$, assume the former. Also $v_j$ is adjacent with at least one of $x$ or $y$. By Lemmas \ref{l1}(1), \ref{l1}(2) and \ref{l1}(3), $\{x,y\} \nsim \{v_{i-1},v_{i+1}, v_{t-1}\}$, $v_{i-1} \nsim v_{t-1}$ and $v_{i-1} \nsim v_{i+1}$. So $\xi(v_i) < 1/2$. Thus this case cannot occur.\\

\noindent\textbf{Subcase 2.3} Assume $v_0$ has exactly one off-cycle neighbour $x$. Let $N(v_0)=\{x,v_1,v_i,v_j,v_k,v_{t-1}\}$ where $1<i<j<k<t-1$. By Lemma \ref{l1}(1) $N(v_0) \cap N(x) \subseteq \{v_i,v_j,v_k\}$. To prove this case we  establishing several useful facts.

\noindent{\bf Fact 1} $x\sim v_i$ implies $v_i\nsim v_{t-1}$. Similarly  $x\sim v_k$ implies $v_k\nsim v_{1}$.\\
\noindent{\bf Proof of Fact 1} We prove the first of these two statements. The second statement can be proven similarly. Assume, to the contrary, that  $v_i \sim \{x,v_{t-1} \}$. \\
{\bf Case A}  $v_i \sim v_1$. If $i \ne 2$, then $\deg(v_i) =6$ and $N(v_i) = \{x,v_0,v_1, v_{i-1}, $ $v_{i+1}, v_{t-1} \}$. By Lemma \ref{l1}(1), $x \nsim \{v_1,v_{i-1}, v_{i+1}, v_{t-1} \}$ and by Lemmas \ref{l1}(2) and \ref{l1}(3), $v_{i+1} \nsim \{v_{i-1}, v_1 \}$ and $v_{i-1} \nsim v_{t-1}$. Since $ \xi(v_i) \ge 1/2$, it follows that $v_{i-1} \sim v_0$ which is not possible if $ i \ne 2$. Hence $i=2$.

\noindent{\bf Case A(1)}  $j \ne i+1(=3)$.  By Lemma \ref{l1}(1), $x \nsim \{v_1,v_{3},v_{t-1} \}$ and by Lemma \ref{l1}(2), $v_1 \nsim \{v_{3},v_{t-1} \}$. By assumption $v_{3} \nsim v_0$. Since $\xi(v_2) \ge 1/2$ it now follows that $\deg(v_2) =6$. If $v_{3} \nsim v_{t-1} $, then necessarily there is a neighbour of $v_2$ that is a true twin of $v_2$. By Lemma \ref{l2} this implies that $G$ contains a strong induced subgraph isomorphic to $H_l$ and with attachment set $S_l$ for some $1 \le l \le 3$, contrary to the hypothesis. So $v_{3} \sim v_{t-1}$. Moreover since $\xi(v_2) \ge 1/2$, the neighbour of $v_2$ that is not one of $x,v_0, v_{1}, v_{3},v_{t-1}$ must be adjacent with at least four of these vertices. So by Lemma \ref{l1}(1) this neighbour, is on $C$, call it $v_q$. So $v_q$ is adjacent with at least four vertices in $N(v_2)$. Since $\Delta=6$, $q=i+2=4$ or $q=t-2$. If $q=4$, then, by Lemma \ref{l1}(2), $v_q \nsim x$. So $v_q$ is adjacent with all vertices of $N[v_2] -\{x\}$. So $q=j=4$ and $v_q \sim v_1$.  Hence $v_0xv_2v_3v_{t-1} \overleftarrow{C}v_4v_1v_0$ is an extension of $C$, a contradiction. So $q \ne 4$.

Hence $q = t-2$. Since $v_3 \sim v_{t-1}$ it follows from Lemma \ref{l1}(2) that $x\nsim v_q$ and thus $\{v_{t-1},v_0,v_1,v_2,v_{3}\} \subseteq N(v_q)$. So $q=k=t-2$. This now yields the cycle extension $v_0xv_2v_1v_{t-2}\overleftarrow{C}v_{3}v_{t-1}v_0$ of $C$, which is not possible.

\noindent{\bf Case A(2)}  $j=i+1(=3)$. Assume that $v_3\nsim v_{t-1}$. Since $\xi(v_2)\geq\frac{1}{2}$, it now follows from Lemmas \ref{l1}(1), and \ref{l1}(2) that $v_2$ has a neighbour $v_q$ adjacent with four vertices in $N(v_2)$, by an argument similar to Case A(1). If $q=k$, then $v_0$ and $v_2$ are true twins which, by Lemma \ref{l2}, produces a contradiction. So $q \ne k$.  Hence, $v_q\nsim v_0$, and it follows that $ \{x,v_{t-1},v_1,v_2,v_{3}\} \subseteq N(v_q)$. Finally, since $\Delta=6$ and $q\ne k$, we must have exactly one of $q=t-2$ or $q=j+1=4$. In either case, we can argue as before, using the fact that $\Delta=6$ and Lemmas \ref{l1}(1) and \ref{l1}(2),  that $\xi(v_q)< 1/2$, which is not possible.

So $v_3\sim v_{t-1}$. By Lemma \ref{l2}, $v_2$ is not a true twin of $v_0$. So $v_2 \nsim v_k$. Since $\xi(v_0)\geq\frac{1}{2}$, it now follows from this observation and Lemmas \ref{l1}(1) and \ref{l1}(2) that $v_k$ is adjacent to three vertices in $\{x,v_{t-1},v_1,v_3\}$. Assume that $x\sim v_k$. It then follows from Lemma \ref{l1}(2) that $k-1\ne j(=3)$ and $k+1\ne t-1$. We conclude that $N(v_k)=\{x,v_{k-1},v_0,v_{k+1}\}\cup S$ where $S$ is some two-element subset of $\{v_{t-1},v_1,v_3\}$. One can argue using the fact that $\Delta=6$ and Lemmas \ref{l1}(1), \ref{l1}(2), and \ref{l1}(3),  that $\xi(v_k)<\frac{1}{2}$, for each possible choice of $S$. (For example suppose $S=\{v_1, v_{t-1} \}$. Then, $\{x, v_0, v_1\} \nsim \{v_{k-1}, v_{k+1}\}$, $v_{k-1} \nsim v_{k+1}$ and $x \nsim v_1$. So $\xi(v_k) < 1/2$.)

We conclude that $x\nsim v_k$, and by the above that $v_k\sim \{v_{t-1},v_1,v_3\}$. Observe that if $k+1=t-1$, or $k-1=j(=3)$, $C$ has the cycle extensions $v_0xv_2v_1v_k\overleftarrow{C}v_3v_{t-1}v_0$ and $v_2xv_0v_1v_k\overrightarrow{C}v_{t-1}v_3v_2$, respectively. Hence, $k+1\ne t-1$ and $k-1\ne j(=3)$. If $v_{k-1} \sim v_{k+1}$ or $v_{k-1}\sim v_{t-1}$, then $C$ has the extension $v_0xv_2v_1v_kv_{t-1}\overleftarrow{C}v_{k+1}v_{k-1}\overleftarrow{C}v_3v_0$ or $v_0xv_2v_1v_k\overrightarrow{C}v_{t-1}v_{k-1}\overleftarrow{C}v_3v_0$, respectively. So $v_{k-1} \nsim \{v_{k+1},v_{t-1}\}$. Since  $v_1 \sim v_k$, it follows from Lemma \ref{l1}(4) that $v_1 \nsim \{v_{k-1}, v_{k+1} \}$.  By Lemma \ref{l1}(2) $v_1 \nsim \{v_{t-1}, v_3\}$. Also since $\Delta = 6$ and $k+1 \ne t-1$ we have  $v_0 \nsim \{v_{k-1}, v_{k+1} \}$. Hence $\xi(v_k)< 1/2$, a contradiction.

\noindent{\bf Case B} $v_i\nsim v_1$. Suppose first that $j \ne i+1$. Then $v_0 \nsim \{v_{i-1}, v_{i+1}\}$. Moreover, by Lemma \ref{l1}(1), $x \nsim \{v_{i-1}, v_{i+1}, v_{t-1} \}$, by Lemma \ref{l1}(2) $v_{i-1} \nsim v_{t-1}$ and  by Lemma \ref{l1}(3) $v_{i-1} \nsim v_{i+1}$.  This accounts for seven non-edges in $\langle N(v_i) \rangle$. So $\deg(v_i) =6$. Since $\xi(v_i) \ge 1/2$  there is a vertex in $N(v_i)$ which is adjacent with all vertices in $N[v_i]$.  By Lemma \ref{l1}(1) this vertex is on $C$, call it $v_q$.  Thus $v_i$ and $v_q$ are true twins, contrary to Lemma \ref{l2}.

So $j = i+1$. As before, we observe that $x \nsim \{v_{t-1},v_{i-1}, v_{i+1}\}$,  $v_{i-1} \nsim v_{i+1}$ and $v_{i-1} \nsim v_{t-1}$. Also since $i > 2$ we have $v_0 \nsim v_{i-1}$. Since $\xi(v_i) \ge 1/2$ and $\Delta=6$, it follows that $\deg(v_i) =6$. Since $\xi(v_i) \ge 1/2$ and, by Lemma \ref{l2}, $v_i$ has no true twin, there is a neighbour $v_q$ of $v_i$ on $C$, where $q \not\in \{0,i-1,i+1,t-1\}$, such that $v_q$ is adjacent with exactly four of the vertices in $\{x,v_0,v_{i-1},v_{i+1}, v_{t-1}\}$. Since $v_{i-1} \nsim \{x,v_0,v_{i+1}, v_{t-1}\}$ and  $G$ is locally connected with $\Delta=6$, $v_{i-1} \sim v_q$. Also since $\xi(v_i) \ge 1/2$, we have $v_{i+1} \sim v_{t-1}$. We show next that $q \ne j+1(=i+2)$ and $ q \ne t-2$. If $q=j+1$, then $v_0xv_iv_{i+1}v_{t-1} \overleftarrow{C}v_qv_{i-1} \overleftarrow{C}v_0$ is an extension of $C$. If $q=t-2$, then $v_0xv_iv_{t-1}v_{i+1}\overrightarrow{C}v_qv_{i-1}\overleftarrow{C} v_0$ is an extension of $C$. Since $\Delta =6$, we see that $v_q$ does not lie on the path $v_{j+2} \overrightarrow{C}v_{t-3}$. So $q \ne k$ and $v_q$ is on the path $v_2 \overrightarrow{C}v_{i-2}$. Hence $v_q \nsim v_0$. So $v_q \sim \{x,v_i,v_{i-1}, v_{i+1}, v_{t-1}, v_{q-1}, v_{q+1}\}$. Since $\Delta =6$, $q = i-2$. As before we see that $x \nsim \{v_{t-1},v_{q-1}, v_{q+1}(=v_{i-1}), v_j\}$, $v_{i-1} \nsim v_j$, $v_i \nsim v_{q-1}$ and $v_{t-1}\nsim \{v_{i-1}, v_{q-1}\}$. So $\xi(v_q) < 1/2$, a contradiction. This completes the proof of Fact 1. $\Box$

\medskip

\noindent{\bf Fact 2} $x \sim v_j$ implies $v_j \nsim \{v_{1}, v_{t-1}\}$. \\
{\bf Proof of Fact 2} Assume, to the contrary, that $v_j \sim \{x,v_1\}$ or $v_j \sim \{x, v_{t-1}\}$, say the former. \\
{\bf Case A}  $v_i,v_j$ and $v_k$ are consecutive on $C$, i.e., $j-1=i$ and $j+1=k$. \\
By Lemma \ref{l1}(1), $x \nsim \{v_1,v_i,v_k,v_{t-1}\}$, by Lemma \ref{l1}(2), $v_i \nsim v_{t-1}$ and $v_k \nsim v_1$ and by Lemma \ref{l1}(3), $v_1 \nsim v_{t-1}$. Since $\xi(v_0) \ge 1/2$ all remaining pairs of vertices in $N(v_0)$ are adjacent. So $v_i \sim v_k$. This is not possible by Lemma \ref{l1}(3).
So $v_i,v_j$ and $v_k$ cannot be consecutive on $C$.

\noindent{\bf Case B} No two of the vertices $v_i$, $v_j$ and $v_k$ are consecutive on $C$, i.e.,  $j-1 \ne i$ and $j+1 \ne k$. \\
By Lemma \ref{l1}(1), $x \nsim \{v_1,v_{j-1},v_{j+1}\}$, by Lemma \ref{l1}(2), $v_{j+1} \nsim v_1$, by Lemma \ref{l1}(3), $v_{j-1} \nsim v_{j+1}$ and since $\Delta = 6$, $v_0 \nsim \{v_{j-1}, v_{j+1} \}$. Since $\xi(v_j) \ge 1/2$ it follows that $\deg(v_j)=6$ and that there is a vertex in $N(v_j) -\{x,v_0, v_1,v_{i-1},v_{i+1} \}$ that is necessarily a true twin of $v_j$. By Lemma \ref{l2} this is not possible.

\noindent{\bf Case C} Exactly two of the vertices $v_i$, $v_j$ and $v_k$ are consecutive on $C$. \\
{\bf Case C(1)} Assume  that $j-1=i$.  If we consider $N(v_j)$ and use Lemma \ref{l1} and the fact that $\Delta =6$, we see that $x \nsim \{v_1,v_i,v_{j+1}\}$ and $v_{j+1} \nsim \{v_0, v_1,v_i\}$. Since $\xi(v_j) \ge 1/2$, $v_j$ has a  neighbour $v_q \not\in \{x,v_0,v_1,v_{j-1}, v_{j+1}\}$  such that $v_q$ is adjacent with four of the vertices in $\{x,v_0,v_1,v_i,v_{j+1}\}$. By Lemmas \ref{l1}(1),  and \ref{l1}(3), $q \ne t-1$.

Suppose $j+2 \le q <t-1$. Since $\Delta=6$, $q=j+2$.  Note,  by Fact 1, we cannot have $v_q\sim \{x,v_0,v_1\}$. Hence, $v_q\sim \{v_{i},v_{j+1}\}$.

Assume first that $v_q\sim x$.  So $v_q$ is adjacent with exactly one of $v_0$ and $v_1$. Let us first assume that $v_q \sim v_1$. We now determine non-adjacencies in $\langle N(v_q) \rangle$.  By Lemmas \ref{l1}(1), \ref{l1}(2) and \ref{l1}(3),  $x \nsim \{v_1,v_i,v_{j+1},v_{q+1}\}$,  $v_1 \nsim \{v_{j+1}, v_{q+1}\}$, and  $v_{j+1} \nsim \{v_i, v_{q+1}\}$.   So $\xi(v_q) < 1/2$, which is not possible. Hence $v_q \nsim v_1$ and thus $v_q \sim v_0$. So $q=k$. Since $\Delta =6$, and from the case we are in and by Lemmas \ref{l1}(1), and \ref{l1}(2) $x \nsim\{v_i,v_{q-1}, v_{q+1}\}$, $v_{q-1} \nsim \{v_i, v_{q+1}\}$. Also if $q < t-2$, then $v_0 \nsim\{v_{q-1}, v_{q+1}\}$. Since $\xi(v_q) \ge 1/2$, $v_i \sim v_{q+1}$. Thus $v_{q+1}v_i \overleftarrow{C}v_1v_jv_{j+1}v_qxv_0\overleftarrow{C}v_{q+1}$ is an extension of $C$, contrary to our assumption. Thus $q(=k)=t-2$. By Lemma \ref{l1} and the fact that $\Delta =6$, $v_{j+1} \nsim \{x,v_0,v_1,v_i,v_{t-1}\}$ and $v_{t-1} \nsim \{x,v_1,v_i, v_j, v_{j+1}\}$. Let $S=\{v_1,v_i, v_{j+1}, v_{t-1}\}$ and $T=\{v_0,v_j, v_k\}$. Note that $\deg (v_0) = \deg (v_j) = \deg (v_k) =6= \Delta$. Moreover, $N(v_k) \subseteq S \cup T$. Since $x \sim T$, $x \nsim S$, $\Delta = 6$ and $G$ is locally connected, $x$ has no neighbours in $G$ other than those in $T$. Similarly, $N(v_{j+1}) =\{v_j, v_k\}$ and $N(v_{t-1}) =\{v_k, v_0\}$. Since $\langle N(v_0) \rangle$ has seven non-adjacencies, namely $x \nsim \{v_1, v_i, v_{t-1}\}$, $v_1 \nsim \{v_k, v_{t-1}\}$, and $\{v_i, v_j\} \nsim v_{t-1}$ and $\xi(v_0) \ge 1/2$ it follows that $v_i \sim v_1$. So $G$ is isomorphic to the graph $H_5$ with attachement set $\{v_1, v_i\}$.

Hence $v_q \nsim x$. Thus $v_q \sim \{v_0, v_1\}$ and once again $q=k$. We now consider the non-adjacencies in $\langle N(v_0) \rangle$.  Using the previous observations, Lemma \ref{l1}, and the fact that $\Delta =6$ we see that $x \nsim \{v_1,v_i,v_k, v_{t-1} \}$, and $v_{t-1} \nsim \{v_1,v_i, v_j \}$. Since $\xi(v_0) \ge 1/2$ all other pairs of vertices in $\langle N(v_0) \rangle$ are adjacent. Hence $v_k \sim \{v_0,v_1,v_i, v_j, v_{t-1} \}$.  Moreover, $v_k \sim \{v_{k-1}(=v_{j+1}), v_{k+1}\}$ and $v_1 \sim v_i$. Since $\Delta = 6$, $k+1 = t-1$. By Lemma \ref{l1} and the fact that $\Delta =6$, $v_{j+1} \nsim \{x,v_0,v_1, v_i\}$. Also $v_{j+1} \nsim v_{t-1}$ otherwise, $v_0xv_jv_1 \overrightarrow{C} v_iv_kv_{j+1}v_{t-1}v_0$ is an extension of $C$. Since $G$ is locally connected and $\deg(v_0) = \deg(v_j) = \deg(v_k) =6$, it follows that $\deg(x) =\deg(v_{t-1}) =\deg (v_{j+1}) =2$. So $G$ contains $H_4$ as strong induced subgraph with attachment set $S_4 =\{v_1,v_i\}$, contrary to the hypothesis.

Hence $1 < q < i$.  So $v_q \nsim v_0$. Hence $v_q \sim \{x,v_1,v_i,v_j,v_{j+1}\}$. Since $\Delta =6$, $q=i-1$ or $q=2$. By Lemma \ref{l1}, the case we are in,  and the fact that $\Delta =6$ it follows in either case that  $\langle N(v_j) \rangle$ has the following non-adjacencies:  $x \nsim \{v_1,v_i,v_{j+1}\}$, $v_0 \nsim \{v_q, v_{j+1}\}$, $v_1 \nsim \{v_i, v_{j+1}\}$ and $v_i \nsim v_{j+1}$. So $\xi(v_j) < 1/2$, which is not possible.

\noindent{\bf Case C(2)} Assume  $j-1\ne i$ and $j+1=k$. We consider non-adjacencies in $\langle N(v_j) \rangle$. By Lemma \ref{l1} we have $x \nsim \{v_1,v_{j-1},v_k\}$, $v_k \nsim \{v_1,v_{j-1} \}$ and as $\Delta = 6$, $v_0 \nsim v_{j-1}$. So since $\xi(v_j)\geq 1/2$, $\deg(v_j) =6$ and hence there is a neighbour $v_q$ of $v_j$ that has four neighbours in $\{x,v_{k},v_0,v_1,v_{j-1}\}$. Assume first that $k<q \le t-1$. Since $v_{t-1} \nsim \{x, v_{j-1} \}$, $q \ne t-1$. As $\Delta=6$ we have $v_q\nsim v_0$, and hence, $\{x,v_1,v_{j-1},v_j,v_k\} \subseteq N(v_q) $. Since $\Delta=6$, it follows that $q=k+1$.  As before we can argue that $\langle N(v_q) \rangle$ has the following non-adjacencies: $x \nsim \{v_1, v_{j-1}, v_k, v_{q+1}\}$, $v_k \nsim \{v_1, v_{j-1}\}$,  and $v_{q+1} \nsim v_k$. Since $\xi(v_q) \ge 1/2$, $v_{q+1} \sim v_{j-1}$. But now $v_0xv_jv_kv_qv_1\overrightarrow{C} v_{j-1}v_{q+1} \overrightarrow{C} v_0$ is an extension of $C$.

Thus, $1<q<j-1$. Since, by Lemma \ref{l2}, $v_j$ and $v_q$ cannot be twins, $v_q$ is non-adjacent with exactly one vertex in $\{x,v_{k},v_0,v_1,v_{j-1}\}$. Since $v_q$ is also adjacent with $v_j$ and as $\Delta =6$ we see that $q=2$ or $q=j-2$.  From the case we are considering,  by Lemmas \ref{l1}(1), \ref{l1}(2) and \ref{l1}(3) and from the fact that $\Delta = 6$ we have the following non-adjacencies in $\langle N(v_0) \rangle$: $x \nsim \{v_1, v_k, v_{t-1}\}$, $v_1 \nsim \{v_k, v_{t-1}\}$, and $v_j \nsim v_{t-1}$. Also by Fact 1, $v_i$ is not adjacent with both $x$ and $v_{t-1}$. Since $\xi(v_0) \ge 1/2$, it follows that all remaining pairs of vertices in $\langle N(v_0) \rangle$ are adjacent. In particular $v_i \sim \{v_j, v_k\}$. Since $\Delta = 6$, $v_i = v_q$.

Assume first that $v_i \nsim x$. Then $v_i \sim \{v_0,v_1,v_j,v_k, v_{t-1}\}$. Moreover, since $v_i=v_q$, we also have $v_i \sim v_{j-1}$. Since $\Delta=6$, it necessarily follows that $i=2$ and $i=j-2$. Since $\langle N(v_j) \rangle$ has seven non-adjacencies, namely, $x \nsim\{v_1,v_i,v_{j-1}, v_k\}$, $v_k \nsim \{v_1, v_{j-1}\}$ and $v_0 \nsim v_{j-1}$, all other pairs of vertices in $\langle N(v_j) \rangle$ are adjacent. In particular, $v_1 \sim v_{j-1}$. Hence $v_0xv_jv_{j-1}v_1v_iv_k \overrightarrow{C}v_{t-1}v_0$ is an extension of $C$.

Suppose next that $v_i \sim x$ and hence that $v_i \nsim v_{t-1}$. Then $v_i \sim \{x,v_0,v_1,v_j,v_k\}$. Since $v_i$ and $v_j$ are not true twins, by Lemma \ref{l2},  $v_i \nsim v_{j-1}$. So $v_{i+1} \ne v_{j-1}$. Since $\Delta = 6$, it follows that $\deg(v_i)=6$ and that $i=2$. Since $v_{j-1} \nsim \{x,v_0,v_i, v_{j+1}\}$ and  $\langle N(v_j) \rangle$ is connected, $v_{j-1} \sim v_1$. Since $i=2$, this produces a contradiction to Lemma \ref{l1}(2). So this case cannot occur. This completes the proof of Fact 2. $\Box$

\medskip

\noindent{\bf Fact 3}
If $v_j$ is the only common neighbour of $x$ and $v_0$, then $v_i$ is non-adjacent with some vertex in $\{v_1,v_j,v_k,v_{t-1}\}$ and
$v_k$ is non-adjacent with some vertex of $\{v_1,v_i,v_j, v_{t-1}\}$. \\
{\bf Proof of Fact 3} We show that $v_i$ is non-adjacent with some vertex of $\{v_1,v_j,v_k,v_{t-1}\}$. It can be argued in a similar manner that $v_k$ is non-adjacent with some vertex of $\{v_1,v_i,v_j, v_{t-1}\}$. Assume, to the contrary, that $v_i \sim \{v_0,v_1,v_j,v_k,v_{t-1}\}$. Since $\Delta=6$, $i-1=1$ or $i+1=j$. It follows from Lemma \ref{l1}(2) that $i+1\ne j$, and hence $i-1=1$, i.e., $i=2$. We consider non-adjacencies in $\langle N(v_0) \rangle$. By our assumption $ x \nsim \{v_1,v_i,v_k,v_{t-1} \}$ and by Fact 2, $v_j \nsim \{v_1,v_{t-1} \}$. If $k=j+1$, then by Lemma \ref{l1}(2), $v_k \nsim v_1$. Since $ \xi(v_0) \ge 1/2$, it follows that $v_1 \sim v_{t-1}$. By Lemma \ref{l1}(3), this is not possible. So $ k \ne j+1$.

Suppose first that $j-1 > 3$. We now consider the non-adjacencies in $\langle N(v_j) \rangle$. By Lemma \ref{l1}(1) and the fact that $\Delta =6$, $\{x,v_0,v_2\} \nsim \{v_{j-1}, v_{j+1} \}$.  Since $\xi(v_j) \ge 1/2$, it follows that $v_j$ has a neighbour not in $N = \{x,v_0,v_2,v_{j-1},v_{j+1} \}$ that is adjacent with at least four vertices of $N$. By Lemma \ref{l1}(1), such a vertex cannot be an off-cycle neighbour of $v_j$ and hence lies on $C$. Let $v_q$ be this neighbour. By Lemma \ref{l2}, $v_q$ must be adjacent with exactly four vertices of $N$. Since $ v_j \nsim \{v_1,v_{t-1} \}$ either $2 < q < j-1$ or $j+1 < q < t-1$. If $q=k$, then $v_q \nsim x$ and hence $v_q(=v_k) \sim \{v_2,v_{j-1},v_j,v_{j+1}\}$ and $v_{j-1} \sim v_{j+1}$. Since $ \Delta =6$, $k=j+2$. We now consider non-adjacencies in $\langle N(v_2) \rangle$. By an earlier observation, $v_j \nsim \{v_1,v_{t-1}\}$ and since $\Delta =6$, $v_k \nsim v_1$ and $v_{3} \nsim \{v_0,v_j,v_k\}$. Also $v_1 \nsim v_{t-1}$; otherwise, $v_0xv_jv_{j+1}v_{j-1} \overleftarrow{C}v_1v_{t-1} \overleftarrow{C}v_kv_0$ is an extension of $C$. This gives seven non-adjacencies in $\langle N(v_2) \rangle$. Since $\xi(v_2) \ge 1/2$, it follows that $v_{3} \sim \{v_1,v_{t-1} \}$. Thus $v_0xv_jv_{j+1}v_{j-1} \overleftarrow{C} v_{3}v_{t-1} \overleftarrow{C}v_kv_2v_1v_0$ is an extension of $C$, contrary to our assumption. So $q \ne k$, i.e., $v_q \nsim v_0$. Thus $v_q \sim \{x,v_2,v_{j-1},v_j,v_{j+1}\}$. Since $\deg(v_2) \le 6$  we see that $q = 3$. But now there exists eight non-adjacencies in $\langle N(v_j) \rangle$: $\{x,v_0,v_2\} \nsim \{v_{j-1}, v_{j+1}\}$, $v_0 \nsim v_3$, and $x \nsim v_2$. So $\xi(v_j) < 1/2$ which is not possible.

This completes the proof of Fact 3. $\Box$

\medskip

\noindent{\bf Fact 4} $N(v_0) \cap N(x) \ne \{v_j\}$. \\
{\bf Proof of Fact 4} Assume that $v_j$ is the only common neighbour of $x$ and $v_0$. By  this assumption, Lemma \ref{l1}(1) and Fact 2, $x \nsim\{v_1, v_i, v_k, v_{t-1}\}$ and $v_j \nsim \{v_1, v_{t-1}\}$. Since $\xi(v_0) \ge 1/2$, there is only one more non-adjacency in $\langle N(v_0) \rangle$. By Fact 3, it follows that $v_i \nsim v_k$. So $E(\langle N(v_0)\rangle)=\{v_1v_i,v_1v_k,v_1v_{t-1},v_iv_j,v_iv_{t-1},$ $v_jv_k,v_jx,v_kv_{t-1}\}$. By Lemma \ref{l1}(2), $i \ne j-1$ and $k \ne j+1$. So $\deg(v_j) =6$. By Lemma \ref{l1}(1) and the case we are in $x \nsim \{v_i,v_{j-1}, v_{j+1}, v_k \}$. Since $\Delta =6$, $v_0 \nsim \{v_{j-1},v_{j+1} \}$. By the above $v_i \nsim v_k$. Hence as $\xi(v_j) \ge 1/2$ we have $v_{j-1} \sim \{v_i,v_{j+1}, v_k \}$ and $v_{j+1}\sim \{v_i,v_k\}$. Since $\Delta=6$, it follows that $i=j-2=2$ and $k=j+2=t-2$, i.e., $i=2$, $j=4$, $k=6$ and $t=8$. But now  $v_0xv_4v_{3}v_{5}v_6v_{7}v_2v_1v_0$ is an extension of $C$ which is not possible. This completes the proof of Fact 4. $\Box$.

\medskip

\noindent{\bf Fact 5} $v_i \sim x$ implies $v_1 \nsim v_i$ or $v_1\nsim v_{t-1}$ and  $v_k \sim x$ implies $v_k\nsim v_{t-1}$ or $v_1\nsim v_{t-1}$.\\
{\bf Proof of Fact 5} We show that $v_i \sim x$ implies $v_1 \nsim v_i$ or $v_1\nsim v_{t-1}$. It can be shown in a similar manner that $v_k \sim x$ implies $v_k\nsim v_{t-1}$ or $v_1\nsim v_{t-1}$. Assume that $v_i \sim x$ and that $v_1\sim \{v_i,v_{t-1}\}$.  By Lemmas \ref{l1}(2) and \ref{l1}(3), $i-1\ne 1$ and $i+1\ne j$. Moreover, by Lemmas \ref{l1}(1), \ref{l1}(2) and \ref{l1}(3), $x \nsim \{v_1,v_{i-1},v_{i+1} \}$, $v_{i+1} \nsim v_1$, and $v_{i-1} \nsim v_{i+1}$. Also since $\Delta =6$, $v_0 \nsim \{v_{i-1}, v_{i+1} \}$. Since $\xi(v_i) \ge 1/2$ it follows that $v_i$ has a neighbour not in $\{x,v_0,v_1,v_{i-1},v_{i+1} \}$ and this neighbour must be a true twin of $v_i$. By Lemma \ref{l2} this is not possible and completes the proof of Fact 5. $\Box$

\medskip

\noindent{\bf Fact 6} $N(v_0) \cap N(x) \ne \{v_i,v_j,v_k\}$.\\
{\bf Proof of Fact 6} Suppose $N(v_0) \cap N(x) = \{v_i,v_j,v_k\}$. By Lemma \ref{l1}(1) $x \nsim \{v_1, v_{t-1}\}$, by Fact 1, $v_i \nsim v_{t-1}$ and $v_k \nsim v_1$ and by Fact 2, $v_j \nsim \{v_1,v_{t-1} \}$. Since $v_i \sim x$ it follows from Fact 5  that either $v_1 \nsim v_{t-1}$ or $v_i \nsim v_1$. Since $v_k \sim x$ it also follows from Fact 5 that either $v_1 \nsim v_{t-1}$ or $v_k \nsim v_{t-1}$. Since $\xi(v_0) \ge 1/2$ we conclude that $v_1 \nsim v_{t-1}$ and that  $v_j \sim \{v_i,v_k\}$, $v_i \sim \{v_1,v_k\}$ and $v_k \sim v_{t-1}$. Since $\Delta=6$, $i=2$ and $k=t-2$. If $j \ne i+2$, then it follows since $\Delta =6$, that $v_{i+1} \nsim \{v_0,v_j,v_k\}$. By Lemmas \ref{l1}(1) and \ref{l1}(2), $v_{i+1} \nsim \{x, v_1\}$. Hence $v_{i+1}$ is isolated in $\langle N(v_i) \rangle$.  This is not possible since $G$ is locally connected. Thus $j=i+2=4$ and similarly $j = k-2=4$. So $t=8$. By Lemmas \ref{l1}(1) and \ref{l1}(2), $\{x,v_1,v_3,v_5,v_7\}$ is an independent set. Since $\Delta =6$ and $G$ is locally connected, we see that $G = \langle \{x,v_0,v_1, \ldots, v_7\} \rangle \cong F_3$, a contradiction. This completes the proof of Fact 6. $\Box$

\noindent{\bf Fact 7} $N(v_0) \cap N(x) \ne \{v_i,v_j\}$ and $N(v_0) \cap N(x) \ne \{v_k,v_j\}$. \\
{\bf Proof of Fact 7} We prove the first of these two statements since the second statement can be proven in a similar manner. Suppose $N(v_0) \cap N(x) =\{v_i,v_j\}$.  By Lemma \ref{l1}(1), $j \ne i+1$. We know from Fact 5 that either $v_1 \nsim v_{t-1}$ or $v_1 \nsim v_i$. Suppose first that $v_1 \nsim v_{t-1}$.  By Lemma \ref{l1}(1) and our assumption, $x \nsim \{v_1, v_k, v_{t-1} \}$, by Fact 1 $v_i \nsim v_{t-1}$ and by Fact 2 $v_j \nsim \{v_1,v_{t-1} \}$. Since $\xi(v_0) \ge 1/2$, all other pairs of vertices in $N(v_0)$ are adjacent. In particular $v_k \sim \{v_1,v_i,v_j,v_{t-1}\}$ and $v_i \sim \{x,v_1,v_j,v_k\}$. Since $v_k \sim v_1$, it follows from Lemma \ref{l1}(2) that $k \ne j+1$. Thus, since $\Delta =6$, $i=2$ and $k=t-2$. If $j \ne i+2$, then the fact that $\Delta =6$ and Lemmas \ref{l1}(1) and \ref{l1}(2) imply $v_{i+1}  \nsim \{x,v_0,v_1,v_j,v_k\}$. Hence $v_{i+1}$ is isolated in $\langle N(v_i) \rangle$, contrary to the fact that $G$ is locally connected. So $j =i+2=4$. If $k > j+2$, then $\Delta =6$ and Lemmas \ref{l1}(1) and \ref{l1}(2) imply $v_{j+1} \nsim \{x,v_0,v_i,v_{i+1}, v_k\}$. This again contradicts the fact that $G$ is  locally connected. So $k = j+2 =6$. By Lemmas \ref{l1}(1) and \ref{l1}(2) $\{x, v_1,v_3, v_5\}$ and $\{x, v_1, v_3, v_7\}$ are independent sets.   Also $v_5 \nsim v_7$, otherwise, $v_0xv_2v_3v_4v_5v_7v_6v_1v_0$ is an extension of $C$. Since $G$ is connected, locally connected and $\Delta =6$ it follows that $\deg(v_7)=\deg(v_5)=\deg(v_3) = 2$ and $\deg(v_1) = \deg(x)=3$. Thus $G \cong F_4$.

Assume now that $v_1 \sim v_{t-1}$.  By Fact 5, $v_i \nsim v_1$. By Lemma \ref{l1}(1), Fact 1, Fact 2, and Fact 6, we have the following seven non-adjacencies in $\langle N(v_0) \rangle$: $x \nsim \{v_1,v_k,v_{t-1} \}$, $v_i \nsim \{v_1,v_{t-1}\}$ and $v_j \nsim \{v_1,v_{t-1} \}$. So all other pairs of vertices in $N(v_0)$ are adjacent. In particular, $v_k \sim \{v_1,v_i,v_j,v_{t-1}\}$ and $v_i \sim v_j$. So by Lemmas \ref{l1}(1) and \ref{l1}(2), $k \ne j+1$ and $j \ne i+1$. Since $\Delta =6$, $k=t-2$.  By Lemma \ref{l1}(1), and the fact that $\Delta =6$ and Fact 6 we have the following seven non-adjacencies in $\langle N(v_j) \rangle$: $\{x,v_0\} \nsim \{v_{j-1},v_{j+1}\}$, $v_{j+1} \nsim v_i$, $v_{j-1} \nsim v_k$ and $v_k \nsim x$. Since $\xi(v_j) \ge 1/2$ all other pairs of vertices in $N(v_j)$ are adjacent. So $v_{j-1} \sim v_{j+1}$, $v_{j+1} \sim v_k$, and $v_{j-1} \sim v_i$. Since $\Delta =6$, $j=i+2$ and $k=j+2$. Hence $v_{t-1}v_1 \overrightarrow{C}v_{j-1}v_{j+1}v_jxv_0v_kv_{t-1}$ is an extension of $C$, a contradiction.  This completes the proof of Fact 7. $\Box$

\medskip

\noindent{\bf Fact 8} $N(v_0) \cap N(x) \ne \{v_i,v_k\}$.\\
 {\bf Proof of Fact 8} Suppose $x \sim \{v_i,v_k\}$ and thus $x \nsim v_j$. By Fact 5 either $v_1 \nsim v_{t-1}$ or both $v_1 \nsim v_i$ and $v_{t-1} \nsim v_k$.  In the latter case there are, by Fact 1, Lemma \ref{l1}(1) and our case, the following non-adjacencies in $\langle N(v_0) \rangle$: $x \nsim \{v_1,v_j,v_{t-1} \}$, $v_i \nsim \{v_1,v_{t-1} \}$ and $v_k \nsim \{v_1,v_{t-1} \}$. Since $\xi(v_0) \ge 1/2$, all other pairs of vertices in $N(v_0)$ are adjacent. So $v_j \sim \{v_1,v_i,v_k,v_{t-1}\}$, $v_i \sim v_k$ and $v_1 \sim v_{t-1}$. Since $\Delta =6$, $j=i+1$ or $j=k-1$. However, from Lemma \ref{l1}(2),  $j \ne i+1$ and $j \ne k-1$. Hence this case cannot occur.

So $v_1 \nsim v_{t-1}$. In addition, by Lemma \ref{l1}(1), Fact 1 and the case we are in, $x \nsim \{v_1,v_j,v_{t-1}\}$, $v_i \nsim v_{t-1}$ and $v_k \nsim v_1$. Thus, since $\xi(v_0) \ge 1/2$, there is at most one additional pair of non-adjacent vertices in $N(v_0)$.  So either $\langle\{v_1,v_i,v_j\} \rangle$ induces a $K_3$ or $\langle\{v_{t-1},v_k,v_j\} \rangle$ induces a $K_3$. Without loss of generality the former case occurs. The reasoning for the second case is analogous.  Since $v_j \sim v_1$, it follows from Lemma \ref{l1}(2), that $j \ne i+1$. If $i \ne 2$, then by Lemmas \ref{l1}(1), \ref{l1}(2), the fact that $\Delta =6$ and from the case we are considering, it follows that $\langle N(v_i) \rangle$ has the following seven non-adjacencies: $x \nsim \{v_1,v_{i-1}, v_{i+1},v_j\}$, $v_1 \nsim v_{i+1}$, $v_0 \nsim \{v_{i-1}, v_{i+1} \}$. Since $\xi(v_i) \ge 1/2$ all other pairs of vertices in $N(v_i)$ are adjacent. In particular $v_{i-1} \sim v_{i+1}$.  By Lemma \ref{l1}(3), $C$ is extendable, a contradiction.

So $i=2$.  Since $v_j \sim v_1$, it follows from Lemma \ref{l1}(2), that $j \ne k-1$. So if $v_j \sim \{v_k, v_{t-1} \}$, then,  $\deg (v_j ) \ge 7$ which is not possible. So $v_j$ is adjacent with at most one of $v_k$ and $v_{t-1}$. Since we have already described six non-adjacencies in $\langle N(v_0) \rangle$ and  $\xi(v_0) \ge 1/2$, $v_j$ is non-adjacent with exactly one of $v_k$ and $v_{t-1}$.  So $\deg (v_j) =6$. Since $\xi(v_0) \ge 1/2$, $v_k \sim \{v_2,v_{t-1} \}$. If $v_j \sim v_{t-1}$, then $v_j \nsim v_k$. But now $\langle N(v_2) \rangle$ has at least seven non-adjacencies, namely:  $x \nsim \{v_{1},v_{3}, v_j\}$, $v_{3} \nsim \{v_0,v_{1}\}$, and $v_k \nsim \{v_{1},v_j \}$. Since $\xi(v_2) \ge 1/2$ all other pairs of vertices in $N(v_2)$ are adjacent. So $v_3 \sim \{v_j, v_k\}$. Since $\Delta =6$, $j=i+2=4$ and $k=t-2$.
If $v_3 \sim v_{t-1}$, then $v_{t-1}v_3v_{t-2} \overleftarrow{C}v_4v_1v_2xv_0v_{t-1}$ is an extension of $C$. But now we can argue as before that $\langle N(v_{t-2}) \rangle$ has at least seven non-adjacencies, namely: $x \nsim\{v_3,v_{t-1}, v_{t-3}\}$, $v_3 \nsim \{v_0,v_{t-1}\}$ and $v_{t-3} \nsim \{v_{t-1}, v_0\}$. Since $\xi(v_{t-2}) \ge 1/2$ it follows that all other pairs of vertices in $N(v_{t-2})$ are adjacent. Hence $v_2 \sim v_{t-3}$, contrary to the fact that $\deg (v_{2}) \le 6$.

Thus $v_j \nsim v_{t-1}$ and $v_j \sim v_k$. Since $\Delta =6$, $k=t-2$ or $k=j+1$. By  the above $k \ne j+1$. So $k=t-2$. Since $\Delta =6$ and $\langle N(v_2) \rangle$ has seven non-adjacencies, namely:  $x \nsim\{v_{1},v_{3}, v_j\}$, $v_{3} \nsim \{v_0,v_{1}\}$, $v_k \nsim \{v_{1}, v_{3} \}$, we have $j=i+2=4$. As $G$ is locally connected and $\Delta =6$, the vertices $v_0,v_2,v_4$, and $v_{t-2}$ all have degree $6$. Since $\Delta =6$ and by Lemma \ref{l1}(4) we see that $v_5 \nsim \{v_0,v_1,v_2\}$. Also $v_5 \nsim v_3$; otherwise, $v_0xv_2v_1v_4 \overleftarrow{C} v_3v_5 \overrightarrow{C} v_{t-1}v_0$ is an extension of $C$. Since $G$ is locally connected, $v_5 \sim v_k(=v_{t-2})$. Since $\Delta =6$, $k=j+2=6$ and thus $t=8$. Since $\Delta =6$ and $G$ is locally connected we see that  $v_{3}, v_{5}$, and $v_{7}$ all have degree $2$ in $G$ and the vertices $x$ and $v_1$ both have degree $3$ in $G$. So $G \cong F_4$, contrary to the hypothesis. So this case cannot occur. This completes the proof of Fact 8. $\Box$

\medskip

\noindent{\bf Fact 9}  $N(v_0) \cap N(x) \ne \{v_i\}$ and $N(v_0) \cap N(x) \ne \{v_k\}$.\\
{\bf Proof of Fact 9} We prove the first of these two statements. The second statement can be proven similarly. By Lemma \ref{l1}(1), Fact 1, Fact 5, and the case we are in $\langle N(v_0) \rangle$ has the following non-adjacencies: $x \nsim \{v_1,v_j,v_k,v_{t-1}\}$, $v_i \nsim v_{t-1}$ and either $v_1 \nsim v_{t-1}$ or $v_i \nsim v_1$. Since $\xi(v_0) \ge 1/2$ it follows that there is at most one more non-adjacency in $\langle N(v_0) \rangle$.

\smallskip

\noindent{\bf Case A} Suppose $v_k \nsim v_1$. So all other pairs of vertices in $N(v_0)$ are adjacent. So in particular $v_j \sim \{v_1,v_i,v_k,v_{t-1} \}$ and $v_i \sim v_k$.  By Lemma \ref{l1}(2), and since $v_j \sim v_1$ we have $ j \ne i+1$. Thus $j >3$. Also since $\Delta =6$, $k=j+1$. We consider two cases depending on whether $v_1 \nsim v_i$ or $v_1 \nsim v_{t-1}$.  If $v_1 \nsim v_i$, then $v_1 \sim v_{t-1}$ and $i \ne 2$. From the case we are considering, by Lemma \ref{l1} and using the fact that $\Delta =6$ we have the following non-adjacencies in $\langle N(v_i) \rangle$: $x \nsim \{v_{i-1}, v_{i+1}, v_j,v_k \}$, $v_{i-1} \nsim \{v_0,v_j\}$, and $v_{i+1} \nsim v_0$. Since $\xi(v_i) \ge 1/2$, $v_{i-1} \sim v_{i+1}$. Hence $v_j \overleftarrow{C}v_{i+1}v_{i-1} \overleftarrow{C}v_1v_{t-1} \overleftarrow{C} v_kv_ixv_0v_j$ is an extension of $C$, which is not possible.

Suppose next that $v_1 \nsim v_{t-1}$ and that $v_1 \sim v_i$. Since $\Delta =6$, $i=2$. By Lemma \ref{l1}(2), $v_3 \nsim v_1$ and so $j \ne 3$. As in the previous case we can argue that in  $\langle N(v_2) \rangle$ we have the following non-adjacencies: $x \nsim \{v_1, v_{3}, v_j,v_k \}$, $v_{3} \nsim \{v_0, v_1\}$ and by assumption $v_k \nsim v_1$.  Since $\xi(v_2) \ge 1/2$ it follows that all other pairs of vertices in $N(v_2)$ are adjacent. In particular $v_{3} \sim \{v_j, v_k\}$. Since $\Delta=6$, $j =i+2=4$ and $k=j+1=5$. Thus $v_0xv_2v_{3}v_5 \overrightarrow{C}v_{t-1}v_4 v_1v_0$ is an extension of $C$ which is not possible.

\smallskip

\noindent{\bf Case B} Suppose $v_k \sim v_1$.  As before, we have the following non-adjacencies in $\langle N(v_0) \rangle$: $x \nsim \{v_1,v_j,v_k, v_{t-1}\}$, $v_i \nsim  v_{t-1}$ and either $v_1 \nsim v_{t-1}$ or $v_i \nsim v_1$.

\smallskip

\noindent{\bf Subcase B(1)} Suppose $v_1 \nsim v_{t-1}$. Since $\xi(v_0)\geq\frac{1}{2}$, there can be at most one additional non-adjacency  in $\langle N(v_0)\rangle$.

\smallskip

\noindent{\bf Subcase B(1.1)} $v_j \nsim v_1$. Then $v_k \sim\{v_1, v_i, v_j, v_{t-1}\}$ and $v_i \sim \{v_1,v_j\}$ and $v_j \sim v_{t-1}$. Since $\Delta=6$, it follows that $k=t-2$ or $k=j+1$.
Suppose first that $k=t-2$. Since $\Delta =6$ we also see that either $j=i+1$ or $i=2$. If $j=i+1$, then $v_0xv_i \overleftarrow{C}v_1v_k \overleftarrow{C} v_j v_{t-1} v_0$ is an extension of $C$. So $j \ne i+1$ and $i=2$. By Lemma \ref{l1}(2), $v_{1} \nsim v_{3}$. If $k=j+1$, then $v_0xv_2 \overrightarrow{C}v_jv_{t-1}v_kv_1v_0$ is an extension of $C$.  So $k \ne j+1$. Hence, using Lemma \ref{l1}(1), the fact that $\Delta=6$ and the case we are in there are eight non-adjacencies in $\langle N(v_2) \rangle$, namely, $x \nsim \{v_1,v_{3},v_j,v_k\}$, $v_{3} \nsim \{v_0,v_{1},v_k\}$, and $v_j \nsim v_1$. This is not possible since $\xi(v_2) \ge 1/2$ and $\Delta =6$.

So $ k\ne t-2$ and $k=j+1$. If $j \ne i+1$, it follows since $\Delta =6$ and $v_i \sim v_1$, that $i=2$. Again by Lemma \ref{l1}(2), $v_{1} \nsim v_{3}$. However, then we see as before that there are eight non-adjacencies in $\langle N(v_2) \rangle$, namely: $x \nsim \{v_{1},v_{3}, v_j,v_k\}$, $v_{3} \nsim \{v_0,v_{1},v_k\}$ and $v_j \nsim v_1$. Since $\xi(v_2) \ge 1/2$ and $\Delta = 6$ this is not possible. So $k=j+1$ and $j=i+1$. Thus $v_0xv_i \overleftarrow{C}v_1v_k \overrightarrow{C} v_{t-1} v_j v_0$ is an extension of $C$.

\noindent{\bf Subcase B(1.2)}  $v_j\sim v_1$.  \\
{\bf Subcase B(1.2.1)} $v_1 \nsim v_i$. Thus $i \ne 2$ and $v_j \sim \{v_i, v_k, v_{t-1}\}$ and $v_k \sim \{v_1,v_i,v_{t-1}\}$. Since $\Delta=6$, it follows that $k=j+1$. Since $v_j \sim v_1$, it follows from Lemma \ref{l1}(2), that $j \ne i+1$. By Lemma \ref{l1}(1) and the fact that $\Delta =6$, $v_{i-1} \nsim \{x,v_0,v_j,v_k\}$. Since $\langle N(v_i) \rangle$ is connected, $v_{i-1} \sim v_{i+1}$. If $j \ne i+2$, then $v_{i+1} \nsim \{x,v_0,v_j,v_k\}$. But then $\langle N(v_i) \rangle$ has eight non-adjacencies, contrary to the fact that $\xi(v_i) \ge 1/2$. So $j=i+2$. Hence $v_0xv_iv_{i+2}v_{i+1}v_{i-1} \overleftarrow{C}v_1v_{j+1}\overrightarrow{C}v_{t-1}v_0$ is an extension of $C$, a contradiction.

\smallskip

\noindent{\bf Subcase B(1.2.2)} $v_1 \sim v_i$. Since $v_j \sim v_1$, it follows from Lemma \ref{l1}(2), that $j \ne i+1$.

\noindent {\bf Subcase B(1.2.2.1)} $v_i \sim \{v_j, v_k\}$. Since $\Delta=6$ it follows that $i=2$.  So, by Lemma \ref{l1}(4), $k \ne j+1$. Also $\Delta=6$ implies that $v_j$ is non-adjacent with at least one of $v_k$ and $v_{t-1}$. If $v_j \nsim v_k$, then there exist seven non-adjacencies in $\langle N(v_i) \rangle$, namely, $x \nsim \{v_1, v_3, v_j, v_k\}$, $v_3 \nsim \{v_0, v_1\}$ and $v_j \nsim v_k$. Moreover, there are seven non-adjacencies in $\langle N(v_0) \rangle$. So $\{v_3,v_{t-1}\} \sim \{v_j, v_k\}$. Hence $\Delta =6$ implies $k=t-2$ and $j=i+2=4$. Thus  $\langle N(v_4) \rangle$ has the following six non-adjacencies: $v_0 \nsim \{v_3, v_5\}$, $v_1 \nsim \{v_3, v_{t-1}\}$ and $v_2 \nsim \{v_5, v_{t-1}\}$. If $v_3 \sim v_{t-1}$, then $v_0xv_2v_1v_4 \overrightarrow{C}v_{t-2}v_3v_{t-1}v_0$ is an extension of $C$, which is not possible. If $v_5 \sim v_{t-1}$, then $v_0xv_2v_1v_4v_3v_{t-2} \overleftarrow{C}v_5v_{t-1}v_0$ is an extension of $C$, which is not possible. So $\xi(v_4) < 1/2$, contrary to the hypothesis.

So $v_j \sim v_k$ and $v_j \nsim v_{t-1}$. So $\langle N(v_0) \rangle$ has the following seven non-adjacencies: $x \nsim \{v_1, v_j, v_k, v_{t-1}\}$ and $v_{t-1} \nsim \{v_1, v_i, v_j\}$. Since $\xi(v_0) \ge 1/2$ all other pairs of vertices in $\langle N(v_0) \rangle$ are adjacent. So $v_k \sim v_{t-1}$. Since $k \ne j+1$ and $\Delta = 6$, it follows that $k=t-2$. By Lemma \ref{l1}(1) and \ref{l1}(2), the fact that $\Delta =6$ and from the case we are in, $\langle N(v_2) \rangle$ has the following non-adjacencies: $x \nsim \{v_1,v_3, v_j,v_k\}$, $v_3 \nsim \{v_0, v_1, v_k\}$. Since $\xi(v_2) \ge 1/2$, $v_3 \sim v_j$. So $\Delta =6$ implies $j=4$. We now consider the non-adjacencies in $\langle N(v_k) \rangle$. If $k \ne j+2=6$, then $v_{k-1} \nsim \{v_0, v_1, v_2, v_4\}$. Since $G$ is locally connected, $v_{k-1} \sim v_{t-1}$. So $v_0xv_2 \overrightarrow{C}v_{k-1}v_{t-1}v_{t-2}v_1v_0$ is an extension of $C$. So $k=j+2=6$ and $t=8$. If $v_3 \sim v_5$, then $v_0xv_2v_3v_5v_4v_1v_6v_7v_0$ is an extension of $C$. Also, $v_3 \nsim v_{t-1}$, otherwise, $v_0xv_2v_3v_7v_6v_5v_4v_1v_0$ is an extension of $C$. Since $\Delta=6$ and by Lemma \ref{l1}(2), $v_3 \nsim\{v_0,v_1\}$. We can argue similarly that $v_5 \nsim \{v_0,v_1,v_2,v_3,v_7\}$. By Lemma \ref{l1}(2), $x \nsim \{v_3, v_5\}$. Since $G$ is locally connected and $\Delta =6$, no vertex in $V(C) \cup\{x\}$ is adjacent with a vertex not in this set. Thus $G \sim F_3$, which is not possible.

\smallskip

\noindent{\bf Subcase B(1.2.2.2)} $v_i \sim v_j$ and $v_i \nsim v_k$. Since $\xi(v_0) \ge 1/2$, it follows that $v_j \sim v_k$ and  $v_{t-1} \sim \{v_j, v_k\}$. Thus $\Delta = 6$ implies that $j =k-1$. Since $v_1 \sim \{v_j,v_k\}$, it follows from Lemma \ref{l1}(4), that $i \ne 2$. Since $\Delta =6$ and by Lemma \ref{l1}(1) we now have seven non-adjacencies in $\langle N(v_i) \rangle$, namely, $x \nsim \{v_1, v_{i-1}, v_{i+1}, v_j\}$, $v_0 \nsim \{v_{i-1}, v_{i+1}\}$ and $v_j \nsim v_{i-1}$. Since $\xi(v_i) \ge 1/2$, $v_{i+1} \sim \{v_j, v_{i-1}\}$. Since $\Delta =6$, $j-1=i+1$. So $v_0xv_iv_{i+1}v_{i-1} \overleftarrow{C} v_1v_j \overrightarrow{C}v_0$ is an extension of $C$, which is not possible.

\smallskip

\noindent{\bf Subcase B(1.2.2.3)} $v_i \sim v_k$ and $v_i \nsim v_j$. Since $\xi(v_0) \ge 1/2$, it follows that $v_k \sim \{v_i, v_j, v_{t-1}\}$ and $v_j \sim v_{t-1}$.  Since $\Delta =6$, $k=t-2$ or $k=j+1$. If $i \ne 2$, then by Lemmas \ref{l1}(1) and \ref{l1}(2) and the fact that $\Delta =6$, we have the following eight non-adjacencies in $\langle N(v_i) \rangle$: $\{x, v_0, v_k\} \nsim \{v_{i-1}, v_{i+1}\}$, $x \nsim v_k$ and $v_1 \nsim v_{i+1}$. So $\xi(v_i) < 1/2$, contrary to the hypothesis. So $i=2$. Hence, by Lemma \ref{l1}(4), $k \ne j+1$. So $k=t-2$. From the case we are in, and the fact that $\Delta =6$ and by Lemma \ref{l1}(4) we have the following six non-adjcencies in $\langle N(v_j) \rangle$: $\{v_0, v_1\} \nsim \{v_{j-1}, v_{j+1}\}$, $v_k \nsim v_{j-1}$ and $v_1 \nsim v_{t-1}$.

If $k \ne j+2$, we also have $v_k \nsim v_{j+1}$. Since $\xi(v_j) \ge 1/2$ it now follows that $v_{t-1} \sim \{v_{j-1}, v_{j+1}\}$. Hence $v_0xv_2 \overrightarrow{C}v_{j-1}v_{t-1}v_{j+1} \overrightarrow{C}v_kv_jv_1v_0$ is an extension of $C$, a contradiction. So $k=j+2$. As before we have the following six non-adjacencies in $\langle N(v_j) \rangle$: $\{v_0, v_1\} \nsim \{v_{j-1}, v_{j+1}\}$, $v_k \nsim v_{j-1}$ and $v_1 \nsim v_{t-1}$. Since $\xi(v_j) \ge 1/2$ either $v_{j-1} \sim v_{t-1}$ or $v_{j+1} \sim v_{t-1}$. If $v_{j-1} \sim v_{t-1}$, then $v_0xv_2 \overrightarrow{C} v_{j-1}v_{t-1}v_j v_{j+1}v_{j+2}v_1v_0$ is an extension of $C$ which is not possible and if $v_{j+1} \sim v_{t-1}$, then $v_0xv_2 \overrightarrow{C}v_jv_1v_{j+2}v_{j+1}v_{t-1}v_0$ is an extension of $C$ which is not possible.


\noindent{\bf Subcase B(2)} $v_1 \nsim v_i$ and $v_1 \sim v_{t-1}$. So $i \ne 2$ and, by Lemma \ref{l1}(3), $j \ne i+1$. \\
{\bf Subcase B(2.1)} $v_i \sim \{v_j, v_k\}$. Then $\langle N(v_i) \rangle$ has the following six non-adjacencies: $x \nsim\{v_{i-1}, v_{i+1}, v_j, v_k\}$, and $v_0 \nsim \{v_{i-1}, v_{i+1}\}$. So $v_{i-1}$ is adjacent with $v_j$ or $v_k$. Suppose first that $k=j+1$. If  $v_{i-1} \sim v_j$, then $v_jv_{i-1}\overleftarrow{C}v_1v_k\overrightarrow{C}v_{t-1}v_0xv_i\overrightarrow{C}v_j$ is an extension of $C$ which is not ossible.  If $v_{i-1} \nsim v_j$, then $\xi(v_i) \ge 1/2$ implies that all remaining edges in $\langle N(v_i) \rangle$ are present. So $v_k \sim \{v_{i-1} , v_{i+1}\}$.  So $\deg(v_k) > 6= \Delta$, a contradiction.   So $k \ne j+1$. Suppose now that $v_{i-1} \sim v_j$. If $v_j \sim v_1$, then $\Delta =6$ implies $v_j \nsim \{v_k, v_{t-1}\}$. Thus $\xi(v_0) < 1/2$, contrary to the hypothesis. Thus $v_j \nsim v_1$. Since $\xi(v_0) \ge 1/2$ it follows that $v_j \sim \{v_k, v_{t-1}\}$. However then $\deg (v_j) \ge 7$.

So $v_{i-1} \nsim v_j$ and $v_{i-1} \sim v_k$.  Since $\xi(v_i) \ge 1/2$, it follows that $v_{i+1} \sim \{v_j, v_k\}$, $v_{i+1} \sim v_{i-1}$ and $v_k \sim \{v_{i-1},v_j\}$. So $\{v_0,v_{i-1},v_i,v_{i+1},v_j, v_{k-1},v_{k+1}\} \subseteq N(v_k)$, contrary to the fact that $\Delta =6$.

\smallskip

\noindent{\bf Subcase B(2.2)} $v_i$ is non-adjacent with exactly one of $v_j$ or $v_k$. Suppose $v_i \nsim v_j$ and $v_i \sim v_k$. Thus $v_j \sim \{v_k, v_{t-1}\}$ and $v_k \sim v_{t-1}$. Since $\Delta =6$, $k=j+1$ or $k=t-2$. Moreover, there are seven non-adjacencies in $\langle N(v_i) \rangle$, namely, $x \nsim \{v_{i-1}, v_{i+1}, v_k\}$ and $\{v_0, v_k\} \nsim \{v_{i-1}, v_{i+1}\}$. Since $\xi(v_i) \ge 1/2$, there is a vertex in $N(v_i) -\{x,v_0,v_{i-1},v_{i+1},v_k\}$ that is necessarily a true twin of $v_i$. By Lemma \ref{l2} this is not possible.

So $v_i \nsim v_k$ and $v_i \sim v_j$. Since $\xi(v_0) \ge 1/2$,  $v_j \sim \{v_1,v_k, v_{t-1}\}$ and $v_k \sim v_{t-1}$. Since $j \ne i+1$ and $\Delta =6$, $k=j+1$. Since $\Delta =6$, $v_j \nsim v_{i-1}$.  So $\langle N(v_i) \rangle$ has the following non-adjacencies: $x \nsim \{v_{i-1}, v_{i+1}, v_j\}$, $v_{i-1} \nsim \{v_0, v_j\}$ and $v_0 \nsim v_{i+1}$. Since $\xi(v_i) \ge 1/2$, there is a vertex $v_q$ in $N(v_i) - \{x, v_0, v_{i-1}, v_{i+1}, v_j\}$. By Lemma \ref{l2}, $v_i$ and $v_q$ are not true twins. So $v_q$ is adjacent with all except exactly one vertex of $N(v_i) - \{ v_q\}$. Since $q \not\in \{1,k, t-1\}$ and $\Delta =6$, $v_0 \nsim v_q$. So $v_q \sim \{v_j,x\}$. Since $\Delta = 6$, $q=j-1$. Since $v_j \sim v_1$, we have a contradiction to Lemma \ref{l1}(2). So this case does not occur.

This completes the proof of Fact 9. $\Box$.
\medskip

\noindent Our result now follows from Facts 4, 6, 7, 8, and 9.
\end{proof}

\begin{corollary}\label{weakly_pancyclic}
If $G$ is a connected locally connected graph with $2 \le \Delta \le 6$ and minimum clustering coefficient at least 1/2, then $G$ is weakly pancyclic.
\end{corollary}
\begin{proof} We have already observed that this is the case if $2 \le \Delta \le 4$. Let $G$ have maximum degree $5$ or $6$. By Theorem \ref{t1}, $G$ is either fully cycle extendable, or $G$ is isomorphic to an $F_i$ for $i \in \{1,2,3,4\}$, or $G$ contains $H_i$ as a strong induced subgraph with attachment set $S_i$ for $i \in \{1,2,3,4,5\}$. We proceed by induction on the order of $G$. Let $G$ be a graph of order 6 or 7 that is connected, locally connected  with $\Delta =5$ or $6$, respectively and minimum clustering coefficient at least 1/2. If $G$ is fully cycle extendable, then $G$ is weakly pancyclic. We observe that $F_1$ and $F_2$ are weakly pancyclic and that $G$ is not isomorphic to $F_i$ for $i \in \{3,4\}$.  So if $G$ is not fully cycle extendable, then $G$ contains $H_i$ as strong induced subgraph with attachment set $S_i$ for some $i \in \{1,2,3\}$ and hence $G$ is isomorphic to $H_i$ for some $i \in \{1,2,3\}$. In each case $G$ is readily seen to be weakly pancyclic.

Suppose now that $G$ is a graph of order $n > 7$ and that every connected locally connected graph of order $k$, $5 \le k <n$, and maximum degree $\Delta$ where $ 2 \le \Delta \le 6$ and minimum clustering coefficient at least 1/2,  is weakly pancyclic. Let $G$ be a connected locally connected graph with $5 \le \Delta \le 6$ and minimum clustering coefficient at least 1/2. If $G$ is fully cycle extendable, then $G$ is weakly pancyclic. Moreover if $G$ is isomorphic to $F_i$ for some $i$, $ i \in \{3,4\}$, then $G$ is readily seen to be weakly pancyclic. Assume thus that $G$ contains an $H_i$ with attachment set $S_i$ as strong induced subgraph. Then $G$ contains a vertex $v$ of degree $2$. It can be shown in a straightforward manner that $G-v$ is a connected locally connected graph with $4 \le \Delta \le 6$ and minimum clustering coefficient at least 1/2. So $G-v$ is weakly pancyclic. Moreover, the circumference of $G-v$ is either $c(G)$ or $c(G) -1$ since the neighbours of every vertex of degree 2 are necessarily adjacent. Since the girth of both $G$ and $G-v$ is $3$, the result now follows.
\end{proof}

\medskip

\noindent{\bf Remark:} The conclusion of Corollary \ref{weakly_pancyclic} still holds if $G$ is disconnected and each component of $G$ has order at least $3$.

\section{Concluding Remarks}
In this paper we added more supporting evidence to Ryj\'{a}\v{c}ek's conjecture: which states that every locally connected graph is weakly pancyclic. We showed that every locally connected graph with minimum clustering coefficient at least 1/2 and maximum degree at most $6$ is weakly pancyclic. Indeed we showed that these graphs, in general, have an even richer cycle structure.  We obtained a complete characterization (in terms of a family of strong induced subgraphs) of these graphs that are fully cycle extendable. It remains an open problem to determine whether the problem, of deciding if a locally connected graph with minimum clustering coefficient at least 1/2 is hamiltonian, is NP-complete. Hendry \cite{H1} conjectured that all hamiltonian chordal graphs are fully cycle extendable. This was shown to be true for several subclasses of the chordal graphs, see \cite{AS, CFGJ}. Recently, however, it was shown in \cite{LS} that this conjecture is not true. These results, and the fact that all hamiltonian locally connected graphs with minimum clustering coefficient at least 1/2 and maximum degree at most $6$ are fully cycle extendable,  prompt the question: which hamiltonian locally connected graphs with minimum clustering coefficient at least 1/2 are fully cycle extendable?

\end{document}